\documentclass{TPmod}
\usepackage{url}

\DeclareMathOperator{\HC}{\mathsf{HC}}
\DeclareMathOperator{\HP}{\mathsf{HP}}
\DeclareMathOperator{\QH}{\mathsf{QH}}
\DeclareMathOperator{\iii}{\texttt{i}}

\newcommand{\vhs}{$\mathsf{VSHS}$}
\newcommand{\power}[1]{[\![ #1 ]\!]}
\newcommand{\laurents}[1]{(\!( #1 )\!)}
\newcommand{\cM}{\mathcal{M}}

\newcommand\ainf{A_{\infty}}

\newcommand{\dbdg}[1]{D^b_{\mathsf{dg}}Coh(#1)}
\newcommand{\BsideOC}{\tilde{\mathfrak{I}}}

\begin{document}

\title{Mirror symmetry: from categories to curve counts}

\author[Ganatra, Perutz and Sheridan]{Sheel Ganatra\footnote{Supported by an NSF postdoctoral fellowship.}, Timothy Perutz\footnote{Partially supported by NSF grants DMS-1406418 and CAREER 1455265.}  and Nick Sheridan\footnote{Partially supported by the National Science Foundation through Grant number
DMS-1310604, and under agreement number DMS-1128155. 
Any opinions, findings and conclusions or recommendations expressed in this material are those of the author and do not necessarily reflect the views of the National Science Foundation.}}

\address{Sheel Ganatra, Department of Mathematics, Stanford University, 450 Serra Mall, Stanford CA 94305-2125, USA.}
\address{Timothy Perutz, University of Texas at Austin, Department of Mathematics, RLM 8.100, 2515 Speedway Stop C1200, Austin, TX 78712, USA.}
\address{Nick Sheridan, Department of Mathematics, Princeton University, Fine Hall, Washington Road, Princeton, NJ 08544-1000, USA.}

\begin{abstract}
{\sc Abstract:} 
We work in the setting of Calabi-Yau mirror symmetry. 
We establish conditions under which Kontsevich's homological mirror symmetry (which relates the derived Fukaya category to the derived category of coherent sheaves on the mirror) implies Hodge-theoretic mirror symmetry (which relates genus-zero Gromov-Witten invariants to period integrals on the mirror), following the work of Barannikov, Kontsevich and others. 
As an application, we explain in detail how to prove the classical mirror symmetry prediction for the number of rational curves in each degree on the quintic threefold, via the third-named author's proof of homological mirror symmetry in that case; we also explain how to determine the mirror map in that result, and also how to determine the holomorphic volume form on the mirror that corresponds to the canonical Calabi-Yau structure on the Fukaya category. 
The crucial tool is the `cyclic open-closed map' from the cyclic homology of the Fukaya category to quantum cohomology, defined by the first-named author in \cite{Ganatra2015}. 
We give precise statements of the important properties of the cyclic open-closed map: it is a homomorphism of \emph{variations of semi-infinite Hodge structures}; it respects \emph{polarizations}; and it is an isomorphism when the Fukaya category is \emph{non-degenerate} (i.e., when the open-closed map hits the unit in quantum cohomology). 
The main results are contingent on works-in-preparation \cite{Perutz2015a,Ganatra2015a} on the symplectic side, which establish the important properties of the cyclic open-closed map in the setting of the `relative Fukaya category'; and they are also contingent on a conjecture on the algebraic geometry side, which says that the cyclic formality map respects certain algebraic structures.
\end{abstract}

\maketitle

\tableofcontents

\section{Introduction}

\subsection{Standing notation}
\label{subsec:setup}

We denote $\BbK_A := \C \laurent{Q}$; $\cM_A := \spec \BbK_A$ will be called the \emph{K\"{a}hler moduli space}. 
We write $T \cM_A := \deriv_\C \BbK_A$ and $\Omega^1\cM_A := \Omega^1(\cM_A/\spec \C)$.
Similarly we denote $\BbK_B := \C \laurent{q}$; $\cM_B := \spec \BbK_B$ will be called the \emph{complex structure moduli space}.

Let $(X,\omega)$ be a connected $2n$-dimensional integral symplectic Calabi-Yau manifold (i.e., $[\omega] \in H^2(X;\Z)$ and $c_1(TX) = 0$). 

Let $Y \to \mathcal{M}_B$ be a smooth, projective, connected scheme of relative dimension $n$, with trivial relative canonical sheaf; and assume that $Y$ is \emph{maximally unipotent}, in the sense of \cite[Definition 1.4]{Perutz2015}. 
The latter condition deserves some explanation: it means that $\mathsf{KS}(\partial_q)^n \neq 0$, where
\[ \mathsf{KS}: \mathcal{T} (\cM_B/\spec \C) \to H^1(Y;\mathcal{T} (Y/\cM_B))\]
is the Kodaira--Spencer map (see \S \ref{subsec:Bvhs}), and the power is taken with respect to the wedge product on the space of polyvector fields, $H^\bullet(\wedge \! ^\bullet \mathcal{T} Y)$.

We will consider versions of mirror symmetry that relate the symplectic invariants of $X$ (which will be linear over the Novikov field $\BbK_A$) to the algebraic invariants of $Y$ (which will be linear over $\BbK_B$).

\subsection{Enumerative mirror symmetry in dimension three: curve counts on the quintic}

We consider Calabi-Yau mirror symmetry in the case when the dimension is $n=3$. 
The classic example is when $X=X^5$ is the `quintic threefold': a smooth quintic hypersurface in $\mathbb{CP}^4$, equipped with a K\"{a}hler form $\omega$ whose cohomology class is the hyperplane class $H$. 
The mirror is the quintic mirror family $Y=Y^5$, which is a crepant resolution of $\tilde{Y}^5/G$, where
\[ \tilde{Y}^5 := \left\{ -z_1\ldots z_5 + q \sum_{j=1}^5 z_j^5 \right\} \subset \mathbb{P}^4_{\BbK_B}\]
and $G$ is non-canonically isomorphic to $(\Z/5)^3$ (the group $(\Z/5)^5$ acts on $\mathbb{P}^4_{\BbK_B}$ by multiplying the coordinates $z_j$ by fifth roots of unity, the diagonal action is trivial, and one restricts to the subgroup that preserves the monomial $z_1\ldots z_5$). 

In \cite{Candelas1991}, Candelas, de la Ossa, Green and Parkes brought the new ideas of mirror symmetry into concrete form by formulating a prediction for the quintic threefold, thereby capturing the imagination of the mathematical community.  Let us review their prediction, following \cite[\S  2]{coxkatz}.

The \emph{$A$-model Yukawa coupling} associated to $X$ is the three-tensor
\begin{eqnarray*}
Yuk_A & \in & Sym^3\left(\Omega^1 \cM_A\right),\\
Yuk_A(Q \partial_Q,Q\partial_Q,Q\partial_Q) & := & \langle H,H,H \rangle_{0,3}
\end{eqnarray*}
where the symbol $\langle H,H,H \rangle_{0,3}$ denotes the genus-zero three-point Gromov-Witten invariant. $Yuk_A$ is a power series in $Q$.  We work with symplectic Gromov-Witten invariants as in \cite{Ruan1995,mcduffsalamon}: so this is a count of pseudoholomorphic maps $u:\mathbb{CP}^1 \to X$, weighted by $Q^{\omega(u)} \in \BbK_A$. 
It can be rewritten as
\[ \langle H,H,H \rangle_{0,3} = \int_X \omega^3 + \sum_{d =1}^\infty n_d \cdot d^3 \cdot \frac{Q^{d}}{1-Q^{d}},\]
where $n_d$ is interpreted as `the virtual number of degree-$d$ curves on $X$' by the Aspinwall-Morrison formula (see for instance \cite{Voisin1996, Bryan2001}).

On the other hand, one defines the \emph{$B$-model Yukawa coupling} associated to $Y$ using Hodge theory: it is the three-tensor
\begin{eqnarray*}
Yuk_B & \in & Sym^3 \left(\Omega^1 \cM_B \right),\\
Yuk_B \left(q \partial_q,q\partial_q,q\partial_q\right) & := & \int_Y \Omega \wedge \nabla_{q\partial_q}^3 \Omega,
\end{eqnarray*}
where $\Omega$ is a specific choice of a relative holomorphic volume form on the family $Y$. 
Namely, $\Omega$ should be `Hodge-theoretically normalized' in the terminology of \cite{coxkatz} (see \S  \ref{subsec:normvol}). 
This determines $\Omega$ up to multiplication by a complex scalar, which for consistency with \cite[\S  2]{coxkatz} we refer to as `$c_2$'. 

Mirror symmetry predicts the existence of an isomorphism
\begin{equation}
\label{eqn:mirrmap}
 \psi: \cM_A \to \cM_B,
\end{equation}
called the \emph{mirror map}, that respects the Yukawa couplings. 
It also predicts a formula for the mirror map: the corresponding isomorphism $\psi^*: \BbK_B \to \BbK_A$ sends $Q \mapsto Q(q)$.
We remark that the constant $c_2$ can be normalized so that the leading terms of the Yukawa couplings match up (the leading term on the $A$-side is $\int_X \omega^3$, which for the quintic is $5$).

The B-model Yukawa coupling can be obtained as follows. One first identifies the Picard-Fuchs (PF) equation satisfied by the periods of some specified holomorphic volume form $\Omega$ with respect to the coordinate $q$  on $\mathcal{M}_B$. The Yukawa coupling for $\Omega$ then satisfies a first-order differential equation related to the PF equations. The constant $c_2$ and the mirror map are determined by solutions to the PF equations which are, respectively, holomorphic and logarithmic at $q=0$. See \cite{coxkatz} for details. We note that the PF equations are not an intrinsic aspect of mirror symmetry, but rather a means of calculation; they will not directly play a role in the present paper.

The $B$-model Yukawa coupling and the mirror map can be explicitly computed: so mirror symmetry gives a prediction for the virtual numbers $n_d$ of degree-$d$ curves on the quintic \cite{Candelas1991}. 
These predictions were verified by Givental \cite{Givental1996} and Lian, Liu and Yau \cite {Lian1997}. 
The results are well-known: $n_1 = 2875, n_2 = 609250,\ldots$. 

\subsection{Hodge-theoretic mirror symmetry}
\label{subsec:hodgems}

Morrison \cite{Morrison1993a} formulated the mirror symmetry predictions of \cite{Candelas1991} in terms of Hodge theory (see also \cite{Kontsevich1994, Morrison1997, coxkatz}). To a family $Y \to \cM_B$ (no longer necessarily of dimension three), one can associate a variation of Hodge structures, using classical Hodge theory; somewhat more surprisingly, using the rational Gromov-Witten invariants of $X$, one can cook up a variation of Hodge structures over $\cM_A$. 
In Morrison's formulation, mirror symmetry predicts the existence of an isomorphism of variations of Hodge structures covering the mirror map \eqref{eqn:mirrmap}. 

In fact, what we will consider in this paper is not exactly a variation of Hodge structures in the classical sense. 
Rather, we consider \emph{variations of semi-infinite Hodge structures} (\vhs), as defined by Barannikov \cite{Barannikov2001}. 
A \vhs{} over $\cM$ consists, briefly, of an $\mathcal{O}_\cM\power{u}$-module $\EuE$, equipped with a flat connection
\[ \nabla: T\cM \otimes \EuE \to u^{-1} \EuE.\]
Here $u$ is a formal variable of degree $2$.
A \emph{polarization} for $\EuE$ is a symmetric, sesquilinear, covariantly constant pairing
\[ (\cdot,\cdot): \EuE \times \EuE \to \mathcal{O}_\cM\power{u}\]
with a certain `nondegeneracy' property; see Definition \ref{defn:prevhs} for the precise definition. 

We will always consider graded polarized \vhs{}, where the base $\cM$ has trivial grading (in fact, we will always assume the base $\cM$ is a formal punctured disc). 
We explain (following \cite[\S  4]{Barannikov2001}) that in this setting, a \vhs{} is equivalent to a $\Z/2$-graded $\mathcal{O}_\cM$-module equipped with a Hodge filtration and a flat connection satisfying Griffiths transversality; and a polarization is equivalent to a covariantly constant pairing on this module, respecting the Hodge filtration in a certain way; see Lemma \ref{lem:realvhs} for the details (this relationship, between bundles with a filtration and equivariant bundles over a formal disc in the $u$-direction, is called the `Rees correspondence').
This is equivalent to the usual notion of a variation of Hodge structures, except that a \vhs{} does not come equipped with an integral structure (i.e., a lattice of flat sections). 
Thus, a \vhs{} is equivalent to a `variation of Hodge structures without the integral structure'. 
We will not consider the integral structure in this paper (however see \cite[\S  2.2.6]{Katzarkov2008} and \cite{Iritani2009a}).

\begin{rmk}
The close relationship between semi-infinite and classical variations of Hodge structures might make one skeptical that the semi-infinite variations deserve their own name. 
However, when the base $\cM$ has non-trivial grading (as happens in the case of mirror symmetry with Fano manifolds on the symplectic side), the relationship between the two is not so close; and semi-infinite variations are the correct notion. 
Many of the results in this paper hold in the Fano case, but only if we use the `semi-infinite' terminology (and incorporate gradings on the base $\cM$); that is why we use it, even though it may be confusing because there is nothing `semi-infinite' happening in the Calabi-Yau case.
\end{rmk}

In \S \ref{subsec:QHnorm}, we define $\EuH^A(X)$, the \emph{$A$-model \vhs{} associated to $X$}; it is a polarized \vhs{} over $\cM_A$.
As a bundle over $\cM_A$, it is trivial with fibre $H^\bullet(X;\C)$, the Hodge filtration is the filtration by degree, the connection is the quantum connection (with connection matrix given by quantum cup product with $[\omega]$), and the polarization is the integration pairing.
In \S \ref{subsec:Bvhs}, we define $\EuH^B(Y)$, the \emph{$B$-model \vhs{} associated to $Y$}; it is a polarized \vhs{} over $\cM_B$.
Its fibre is the relative de Rham cohomology of $Y$ (with the grading collapsed to a $\Z/2$-grading), the Hodge filtration is the usual one, the connection is the Gauss-Manin connection, and the polarization is given by the integration pairing. 

\begin{rmk}
Note that we do not consider the classical polarized variation of Hodge structures associated to a smooth and proper family of varieties, in which the polarization depends on a choice of K\"ahler class. Rather, the polarization is a sign-modified version of the integration pairing. Moreover, the $\Z$-grading on de Rham cohomology, which decomposes the classical Hodge structure into summands of different weights, is here collapsed to a $\Z/2$-grading. 
Perhaps it is helpful to recall that the global Torelli theorem for K3 surfaces has to do with the Hodge structure on $H^2$, whereas the derived Torelli theorem for K3 surfaces \cite{Orlov1997} has to do with the version with the collapsed grading (which was introduced in this context by Mukai \cite{Mukai1987}).
The formulation of (higher-dimensional) Hodge-theoretic mirror symmetry conjectured in \cite[\S 8.6.3]{coxkatz} \emph{does} invoke the classical polarized variation of Hodge structures on $H^\bullet(Y;\C)$ associated with a K\"ahler class, and the Hodge decomposition of $H^\bullet(X;\C)$ arising from a specific complex structure on $X$. We do not know how to incorporate the complex structure into our categorical story. 
\end{rmk}

\begin{defn}
\label{defn:hodgems}
We say that $X$ and $Y$ are \emph{Hodge-theoretically mirror} if there exists an isomorphism
\[ \psi: \cM_A \to \cM_B,\]
and an isomorphism of \vhs{} over $\cM_A$,
\[ \EuH^A(X) \cong \psi^* \EuH^B(Y).\]
\end{defn}

It is well-known (see e.g. \cite{coxkatz,Barannikov2001}) that a \vhs{} $\EuH$ over $\cM$ with a certain `miniversality' property determines canonical coordinates on its base, up to multiplication by a complex scalar: one can think of this as an affine structure on $\cM$. 
We give detailed explanations on this point in \S  \ref{sec:vhs}, adapted to our setup.
The \vhs{} that we consider are miniversal in the appropriate sense, so $\EuH^A(X)$ and $\EuH^B(Y)$ determine canonical coordinates on their respective bases: and in the situation of Hodge-theoretic mirror symmetry, the mirror map $\psi$ must match up these canonical coordinates. 
In particular, $\psi$ is uniquely determined up to multiplication by a complex scalar, which for consistency with \cite[\S  2.5]{coxkatz} we denote by `$c_1$'. 

This prescription gives rise to the aforementioned explicit formula for the mirror map in terms of solutions to the Picard-Fuchs equation, up to the undetermined constant $c_1$. The A- and B-model $n$-fold Yukawa couplings can be computed from the corresponding \vhs{}. Thus Hodge-theoretic mirror symmetry implies equality of Yukawa couplings, up to the constant $c_1$, which can be normalized using the first-order term of the Yukawa coupling (which is $n_1 = 2875$ for the quintic threefold). 
Therefore, Hodge-theoretic mirror symmetry implies equality of Yukawa couplings (up to determination of the constant $c_1$).

When the dimension is not equal to three, Yukawa couplings do not determine the genus-zero Gromov--Witten invariants, but only certain products of these invariants. As we have formulated it, the A-model \vhs{} contains much more information about the genus-zero GW invariants than just the Yukawa couplings, but still not complete information (it would be complete if we considered the `big' A-model \vhs{}, but we have not done that).
So we would like to answer the question: supposing Hodge-theoretic mirror symmetry to be true in the sense of Definition \ref{defn:hodgems}, how much information about Gromov-Witten invariants can one compute from the $B$-model \vhs{}? 

Following the construction of Barannikov \cite{Barannikov2001}, we give a precise answer to this question (see Theorem \ref{thm:hodgemsinfo}): \emph{Hodge-theoretic mirror symmetry allows us to determine the matrix $A(Q)$ of quantum cup-product with $[\omega]$, up to substitution $Q \mapsto Q/c_1$.}

In this paper, we do not address the question of how to compute the corresponding B-side matrix in practice.

\subsection{The Fukaya category}
\label{subsec:fukint}

We consider some version of the Fukaya category of $X$, which we denote $\EuF(X)$. 
We restrict ourselves to versions where $\EuF(X)$ is $\Z$-graded and $\BbK_A$-linear.

\begin{rmk}
One should only expect $\EuF(X)$ to be $\Z$-graded when $X$ is Calabi-Yau, and one should only expect it to be $\BbK_A$-linear (as opposed to being defined over some larger Novikov field) when the symplectic form is integral. 
\end{rmk}

In \S  \ref{sec:fuk}, we give a list of properties that we need the Fukaya category $\EuF(X)$ to have in order for our results to work. 
We expect these properties to hold very generally, so we do not tie ourselves to a particular version of the Fukaya category. 
However, it will be proven in \cite{Perutz2015a,Ganatra2015a} (in preparation) that the \emph{relative Fukaya category} has all of the necessary properties, so the range of proven applicability of our results is not empty (and in fact, includes the very interesting case of Calabi-Yau hypersurfaces in projective space, such as the quintic threefold $X^5$, as we will explain in \S  \ref{subsec:app}). 

Let us briefly outline what the construction of the relative Fukaya category looks like, so the reader can keep a concrete example in mind. 
It depends on a choice of \emph{integral Calabi-Yau relative K\"{a}hler manifold}: that is, a Calabi-Yau K\"{a}hler manifold $(X,\omega)$, together with an ample simple normal crossings divisor $D \subset X$, and a proper K\"{a}hler potential $h$ for $\omega$ on $X \setminus D$: in particular, $\omega = d\alpha$ is exact on $X \setminus D$, where $\alpha := d^c h$. 
This defines a map
\begin{eqnarray*}
H_2(X,X \setminus D) & \to & \R, \\
u & \mapsto & \omega(u) - \alpha(\partial u),
\end{eqnarray*}
which we require to take integer values (hence the word `integral' in the name).

Objects of the relative Fukaya category are closed, exact Lagrangian branes $L \subset X \setminus D$. 
Floer-theoretic operations are defined by counting pseudoholomorphic curves $u:\Sigma \to X$, with boundary on Lagrangians in $X \setminus D$ (transversality of the moduli spaces is achieved using the stabilizing divisor method of Cieliebak and Mohnke). 
These counts of curves $u$ are weighted by $Q^{\omega(u) - \alpha(\partial u)} \in \BbK_A$. 
Note that these monomials really do lie in $\BbK_A$, by our assumption that the exponent is an integer.
The resulting curved $\ainf$ category is denoted $\EuF(X,D)_{curv}$: we define an honest (non-curved) $\ainf$ category $\EuF(X,D)$, whose objects are objects of $\EuF(X,D)_{curv}$ equipped with bounding cochains.

We emphasise that, if you want to apply our results to your favourite version of the Fukaya category, you just need to verify that it has the properties outlined in \S  \ref{sec:fuk}.

\subsection{Homological mirror symmetry}

Let $X$ and $Y$ be as in \S \ref{subsec:setup}.
Let $\EuF(X)$ a version of the Fukaya category of $X$ as in the previous section, and let $\dbdg{Y}$ be a $\mathsf{dg}$ enhancement of $D^bCoh(Y)$, the bounded derived category of coherent sheaves on $Y$: we regard it as a $\Z$-graded, $\BbK_B$-linear $\ainf$ category. 
The $\mathsf{dg}$ enhancement is unique up to quasi-equivalence, by  \cite[Theorem 8.13]{Lunts2010}.
It is triangulated and split-closed, in the $\ainf$ sense (see \cite[Lemma 5.3]{Seidel2003} for split-closure).

If $\EuC$ is an $\ainf$ category, `$\twsplit \EuC$' denotes the split-closed triangulated envelope (denoted `$\Pi (Tw(\EuC))$' in \cite{Seidel2008})

\begin{defn}
We say that $X$ and $Y$ are \emph{homologically mirror} if there exists an isomorphism $\psi: \cM_A \to \cM_B$, and a quasi-equivalence of $\BbK_A$-linear $\ainf$ categories
\begin{equation} \label{eq:hms}
\twsplit \EuF(X) \cong \psi^* \dbdg{Y}.
\end{equation}
To clarify: since $\EuF(X)$ is $\BbK_A$-linear and $\dbdg{Y}$ is $\BbK_B$-linear, we need the isomorphism $\psi^*:\BbK_B \to \BbK_A$ between their respective coefficient fields in order to compare them.
\end{defn}

In \cite{Kontsevich1994}, Kontsevich conjectured that mirror pairs $(X,Y)$ ought also to be homologically mirror. 
He also conjectured that this `homological mirror symmetry' (HMS) implies Hodge-theoretic and hence enumerative mirror symmetry.
The main result of this paper is about establishing criteria under which the latter claim holds. 

\begin{main}
\label{thm:main}
Suppose that $X$ and $Y$ are as in \S \ref{subsec:setup}, $\EuF(X)$ satisfies the properties outlined in \S  \ref{sec:fuk}, that $X$ and $Y$ are homologically mirror, and furthermore that Conjecture \ref{conj:B} holds. 
Then $X$ and $Y$ are also Hodge-theoretically mirror. That is, there is an isomorphism of \vhs,
\[ \EuH^A(X) \cong \psi^* \EuH^B(Y), \]
with the same mirror map $\psi$ as appears in the statement of homological mirror symmetry.
\end{main}

The proof of Theorem \ref{thm:main} goes via Kontsevich's noncommutative Hodge theory, and its broad outline was no doubt foreseen by Kontsevich, Barannikov and others long ago (see in particular \cite{Barannikov1998,Costello2009,Katzarkov2008}). 

\subsection{\vhs{} from categories}

Let $\BbK \supset \C$ be a field extension of $\C$, and denote $\cM := \spec \BbK$ (in this paper we will consider the case that $\cM$ is a formal punctured disc, so $\BbK \cong \C \laurent{q}$). 
In this section, we summarize a well-known construction that associates, to a $\Z$-graded $\BbK$-linear $\ainf$ category $\EuC$ satisfying certain finiteness and duality conditions, a degenerate version of a \vhs{} over $\cM$, which we call a \emph{pre-\vhs{}} (Definition \ref{defn:prevhs}). 
A pre-\vhs{} over $\cM$ can be thought of as the same data as a \vhs{} over $\cM$, except that the $\mathcal{O}_\cM$-module is not required to be a vector bundle of finite rank, and the polarization (if it exists) is not required to be nondegenerate.
Under certain hypotheses on $\EuC$ which hold in the examples of interest, the resulting pre-\vhs{} is actually a
\vhs.
\begin{defn}
Let $\EuC$ be a $\Z$-graded $\BbK$-linear $\ainf$ category.
\begin{itemize}
\item $\EuC$ is called \emph{ proper} if it has cohomologically finite rank morphism spaces (over $\BbK$).

\item $\EuC$ is called \emph{(homologically) smooth} if the diagonal bimodule
    $\EuC_{\Delta}$ is a perfect bimodule.  
\end{itemize}
\end{defn}

\begin{prop}
\label{prop:vhscat}
Let $\EuC$ be a $\Z$-graded $\BbK$-linear $\ainf$ category. Then
\begin{itemize}
\item The data \[ \left( \HC_\bullet^-(\EuC), \nabla^{GGM} \right )\] 
    forms an unpolarized pre-\vhs, where $\HC_\bullet^-(\EuC)$ is the negative cyclic homology of $\EuC$, and $\nabla^{GGM}: T \cM  \otimes \HC^-_\bullet(\EuC)\to u^{-1} \HC^-_\bullet(\EuC)$ is Getzler's Gauss-Manin connection \cite{Getzler1993}.
\item If $\EuC$ is furthermore proper, and endowed with an $n$-dimensional weak proper Calabi-Yau structure in the sense of Definition \ref{properCY}, then the pre-\vhs{} $(\HC_\bullet^-(\EuC),\nabla^{GGM})$ acquires an $n$-dimensional polarization, given by Shklyarov's higher residue pairing $\langle -,-\rangle_{res}$ \cite{Shklyarov2013}.  

\item     If $\EuC$ is furthermore \emph{smooth}, and the \emph{non-commutative Hodge--de Rham spectral sequence degenerates}, then this polarized pre-\vhs{} is in fact a polarized \vhs.
    \end{itemize}
The data $\left( \HC_\bullet^-(\EuC), \nabla^{GGM} \right )$ and the pairing  $\langle -,-\rangle_{res}$ are \emph{Morita invariant}: categories with quasi-equivalent split-closed triangulated envelopes give rise to the same polarized \vhs{}. 
\end{prop} 

\begin{rmk}
Let us recall Kontsevich--Soibelman's \emph{non-commutative Hodge--de Rham degeneration conjecture} \cite[Conjecture 9.1.2]{Kontsevich2006}: it says that the non-commutative Hodge--de Rham spectral sequence degenerates for arbitrary proper and smooth $\EuC$; if it holds we can remove that hypothesis from the final bullet point.
\end{rmk}

\begin{rmk}
The requirement that $\EuC$ admit an $n$-dimensional weak proper Calabi-Yau
structure is not used in the construction of any of the structures above, but
serves only to ensure that Shklyarov's higher residue pairing is graded symmetric.
\end{rmk}

Much of Proposition \ref{prop:vhscat} appears directly in the literature: in particular, the construction of the connection for $A_\infty$ algebras  is due to Getzler \cite{Getzler1993} (see also \cite{Tsygan2007,Dolgushev2011}), and the construction of the polarization for $\mathsf{dg}$ categories is due to Shklyarov \cite{Shklyarov2013} (the adaptations to $\ainf$ categories are minor). 
The fact that this data together is Morita invariant and satisfies the axioms of a pre-\vhs{} is known or expected and at least partially appears in various sources. 
See the companion note \cite{Sheridan2015a} for a self-contained proof of Proposition \ref{prop:vhscat}, along with an explanation of how our conventions and formulae for these structures on an $\ainf$ category align with existing references.

\subsection{Comparison of \vhs: symplectic side}

Assume that $\EuF(X)$ has all of the properties listed in \S  \ref{sec:fuk}. 
On the symplectic side of mirror symmetry, the key results are the following (proved in \S  \ref{sec:pf}):

\begin{main}
\label{thm:cycoc}
There exists a map of polarized pre-\vhs, called the \emph{negative cyclic open-closed map}:
\begin{equation}
\label{eqn:ocminus}
\widetilde{\EuO\EuC}^-: \HC_\bullet^-(\EuF(X)) \to \EuH^A(X).
\end{equation}
It respects polarizations.
\end{main}

Explicitly, Theorem \ref{thm:cycoc} says that the map $\widetilde{\EuO\EuC}^-$ respects connections, in the sense that 
\[ \widetilde{\EuO\EuC}^- \circ \nabla^{GGM} = \nabla^{QDE} \circ \widetilde{\EuO\EuC}^-,\]
and also that it respects polarizations, in the sense that 
\[ \langle \widetilde{\EuO\EuC}^-(\alpha),\widetilde{\EuO\EuC}^-(\beta) \rangle = \langle \alpha,\beta \rangle_{res}.\]

Now we establish criteria under which the cyclic open-closed map is an isomorphism. 
The crucial hypothesis is called \emph{non-degeneracy} of the Fukaya category, and was introduced in \cite{Ganatra2013}:

\begin{defn}
The Fukaya category $\EuF(X)$ is called \emph{non-degenerate} if the open-closed map 
\[ \EuO\EuC: \HH_\bullet(\EuF(X)) \to \QH^{\bullet+n}(X)\]
hits the unit $e \in \QH^0(X)$.
\end{defn}

\begin{rmk}
    It follows from the definition of Hochschild homology that the preimage
    $[\sigma]$ of $e \in QH^0(X)$ is necessarily contained in the image of the
    inclusion $\HH_{\bullet}(\EuA) \to \HH_{\bullet}(\EuF(X))$, for some finite
    full sub-category $\EuA \subset \EuF(X)$. We call any such $\EuA$ an
    \emph{essential sub-category}; the work of \cite{Abouzaid2010a},
    implemented for relative Fukaya categories in \cite{Perutz2015a}, implies that
    any such $\EuA$ split-generates $\EuF(X)$.
\end{rmk}

\begin{rmk}\label{rmk:automaticsplitgeneration} It follows from \cite[Theorem B]{Perutz2015} that, if $X$
    and $Y$ are homologically mirror, then $\EuF(X)$ is automatically
    non-degenerate (the standing assumption that $Y$ is \emph{maximally unipotent} is crucial for this).
\end{rmk}

\begin{main}[Compare \cite{Ganatra2013, Ganatra2015}]
\label{thm:ociso}
If $\EuF(X)$ is non-degenerate and smooth, then $\widetilde{\EuO\EuC}^-$ is an isomorphism: so `$\EuF(X)$ knows the $A$-model \vhs'. 
\end{main}

\begin{rmk}
It follows from Theorem \ref{thm:ociso} that $\EuF(X)$ satisfies the
non-commutative Hodge--de Rham degeneration conjecture. In particular,
$(\HC_\bullet^-(\EuF(X)),\nabla^{GGM},\langle -,-\rangle_{res})$ is a genuine
polarized \vhs.
\end{rmk}

\subsection{Comparison of \vhs: algebro-geometric side}
\label{subsec:vhsBside}

On the algebro-geometric side, $\dbdg{Y}$ is proper (because $Y$ is proper as a
scheme), smooth (as $Y$ is smooth as a scheme), and admits a weak
proper Calabi-Yau structure (because the canonical sheaf of $Y$ is trivial).  Proposition
\ref{prop:vhscat} therefore endows the negative cyclic homology
$\HC_\bullet^-(\dbdg{Y})$ with the structure of a pre-\vhs{}  (which
is in fact a \vhs{}, as we will see in Remark \ref{rmk:Bhodgederham}). 
There is an intermediate object $\HC_\bullet^-(Y)$, sitting between this one
and $\EuH^B(Y)$, which is the {\it negative cyclic homology of the scheme} of
Loday \cite{Loday1986} and Weibel \cite{Weibel1996}.  It is defined to be the
derived global sections of the sheafification of negative cyclic homology
groups of the structure sheaf of $Y$.

By \cite{Keller1998a}, there is an isomorphism of graded $\BbK_B\power{u}$-modules,
\begin{equation}\label{eq:kellerisomorphism}
    \HC_\bullet^-(\dbdg{Y})  \stackrel{\cong}{\rightarrow} \HC_{\bullet}^-(Y).
\end{equation}
Next, there is a Hochschild-Kostant-Rosenberg (HKR) type isomorphism of graded
$\BbK_B\power{u}$-modules \cite{Weibel1997} 
\[    \tilde{I}_{HKR} \colon \HC_{\bullet}^-(Y ) \to \EuE^B(Y),\]
where
\[\EuE^B(Y) := \bigoplus_{i \in \Z}  u^i \cdot
F^{-i} H_{\mathsf{dR}}^{\bullet-2i}(Y)  \]
is the $\BbK_B \power{u}$-module underlying the $B$-model \vhs{} $\EuH^B(Y)$ (compare \S \ref{subsec:Bvhs}). 

The $B$-model \vhs{} $\EuH^B(Y)$ also comes with a connection, which is $u^{-1}$ times the Gauss--Manin connection, and a polarization, which is the integration pairing (see \S \ref{subsec:Bvhs} for details).

The map induced by $\tilde{I}_{HKR}$ on the associated graded modules of the
$u$-adic filtrations is the HKR isomorphism for Hochschild homology:
\[ I_{HKR}: \HH_\bullet(Y) \to H^\bullet(\Omega^{-\bullet}Y).\] 
However, this isomorphism does not respect the relevant algebraic structures. 
As suggested in  \cite{Caldararu2005} (following \cite{Kontsevich2003}), one should consider instead the `modified' HKR map 
\[ I_K \co \HH_\bullet (Y) \xrightarrow{I_{HKR}} H^\bullet(\Omega^{-\bullet}Y) \xrightarrow{\mathsf{td}^{1/2}(Y) \wedge -} H^\bullet(\Omega^{-\bullet}Y),\] 
where $\mathsf{td}^{1/2}(Y)$ is the square root of the Todd class of $T Y$. 
It was conjectured in \cite[Conjecture 5.2]{Caldararu2005} and proven in \cite{Calaque2012} (respectively, \cite{Markarian2008,Ramadoss2008a}) that this map respect the `calculus' structure (respectively, the Mukai pairing).

Therefore it makes sense to modify $\tilde{I}_{HKR}$ to 
\begin{equation}
\label{eqn:itK} \tilde{I}_K \co \HC_\bullet^-(Y) \xrightarrow{\tilde{I}_{HKR}} \EuE^B(Y) \xrightarrow{\mathsf{td}^{1/2}(Y) \wedge -}\EuE^B(Y),
\end{equation}
where $\mathsf{td}^{1/2}(Y)$ is now treated as a class in $\EuE^B(Y)_0 = \bigoplus_i u^i \cdot F^{-i} H_{\mathsf{dR}}^{-2i}(Y)$.
Combining \eqref{eq:kellerisomorphism} and \eqref{eqn:itK}, we obtain an isomorphism

\begin{equation}\label{eqn:BsideOC}
    \BsideOC: \HC_\bullet^-(\dbdg{Y}) \rightarrow  \EuE^B(Y)
\end{equation}

\begin{conj}
\label{conj:B}
The isomorphism \eqref{eqn:BsideOC} is an isomorphism of \vhs. 
Explicitly, this means:
\begin{enumerate}
    \item \label{it:conn} The map $\BsideOC$ intertwines connections; and
    \item \label{it:pair} The map $\BsideOC$ intertwines pairings.
\end{enumerate} 
\end{conj}

\begin{rmk}
Work of Cattaneo, Felder and Willwacher \cite{Cattaneo2011} (working in the smooth category rather than the category of schemes) goes some way towards verifying part \eqref{it:conn} of Conjecture \ref{conj:B}.
\end{rmk}

\begin{rmk}
Part \eqref{it:pair} of Conjecture \ref{conj:B} is related to C\u{a}ld\u{a}raru's conjecture \cite[Conjecture 5.2]{Caldararu2005}, which says that the associated graded of $\tilde{I}_K$ (namely, $I_K$) intertwines pairings. 
The pairing on $\HH_\bullet(Y)$ is called the \emph{Mukai pairing}. 
C\u{a}ld\u{a}raru's conjecture has been verified by Markarian \cite{Markarian2008} and Ramadoss \cite{Ramadoss2008a}. 
In the cases considered in this paper, the results of Markarian and Ramadoss, combined with part \eqref{it:conn}, suffice to verify part \eqref{it:pair} of the conjecture, by Lemma \ref{lem:pairfix}.
\end{rmk}

\begin{rmk}\label{rmk:Bhodgederham}
    The fact that there is an isomorphism $\BsideOC$ of underlying $\BbK_B\power{u} $ modules
    implies, by reduction to the commutative case, that the non-commutative
    Hodge de Rham spectral sequence degenerates. 
    In particular, the pre-\vhs{} structure on $\HC_\bullet^-(\dbdg{Y})$ is actually a \vhs.
\end{rmk}

\subsection{Proof of Theorem \ref{thm:main}}

We prove Theorem \ref{thm:main}. 
Observe the following diagram:
\[ \xymatrix{\HC_\bullet^-(\twsplit \EuF(X)) \ar[r] \ar[d]^{\widetilde{\EuO\EuC}^-} & \psi^* \HC_\bullet^-(\dbdg{Y}) \ar[d]^{\BsideOC} \\
\EuH^A(X) & \psi^* \EuH^B(Y).} \]
The top arrow is the isomorphism induced by the quasi-equivalence $\twsplit \EuF(X) \cong \dbdg{Y}$ (using the Morita invariance from Proposition \ref{prop:vhscat}).
The left vertical arrow is the composition of $\widetilde{\EuO\EuC}^-$ with the isomorphism $\HC_\bullet^-(\twsplit \EuF(X)) \cong \HC_\bullet^-(\EuF(X))$, again using Morita invariance. 
We observe that $\twsplit \EuF(X)$ is smooth, because $\dbdg{Y}$ is (this follows from the fact that $Y$ is smooth). 
The left vertical arrow is a morphism of polarized pre-\vhs{} by Theorem \ref{thm:cycoc}; and since $\EuF(X)$ is also non-degenerate by hypothesis, it is actually an isomorphism by Theorem \ref{thm:ociso}. 
The right vertical arrow is an isomorphism of $\BbK_A\power{u}$-modules by \cite{Keller1998a,Weibel1997}; the isomorphism respects the polarized \vhs{} structure by Conjecture \ref{conj:B}. 

\subsection{Application: Calabi-Yau hypersurfaces in projective space}
\label{subsec:app}

We consider the mirror pair $(X^n,Y^n)$, where:
\[ X^n := \left\{ \sum_{j=1}^n z_j^n = 0\right\} \subset \mathbb{CP}^{n-1};\]
and $Y^n := \tilde{Y}^n/G$, where
\[ \tilde{Y}^n := \left\{ -z_1\ldots z_n + q \sum_{j=1}^n z_j^n = 0\right\} \subset \mathbb{P}^{n-1}_{\BbK_B}\]
and 
\[ G := \{(\zeta_1,\ldots,\zeta_n): \zeta_j^n =1, \zeta_1\ldots\zeta_n = 1\} /(\zeta,\ldots,\zeta)\]
acts on $\tilde{Y}^n$ by multiplying the coordinates $z_j$ by $n$th roots of unity. 

We recall the following:

\begin{thm}\emph{(\cite[Theorem 1.8]{Sheridan2015})}
\label{thm:hmscy}
$X^n$ and $Y^n$ are homologically mirror. 
Furthermore, the mirror map
\[ \psi^*: \BbK_B \to \BbK_A \]
satisfies $\psi^*(q) = \pm Q + \mathcal{O}(Q^2)$: i.e., the leading-order term is $\pm 1$.
\end{thm}

To be more precise about the statement of Theorem \ref{thm:hmscy}, we must be explicit about which version of the Fukaya category we use, and also which version of $\dbdg{Y^n}$ we use. 
On the Fukaya side, we consider the relative Fukaya category $\EuF(X^n,D)$, where $D \subset X^n$ is the union of the $n$ coordinate hyperplanes. 
This relative Fukaya category was constructed in \cite{Sheridan2015}: the fact that it has all of the additional properties enumerated in \S  \ref{sec:fuk} will be proven in \cite{Perutz2015a,Ganatra2015a}. 

On the coherent sheaves side, we consider the bounded derived category of $G$-equivariant coherent sheaves, $D^b_{\mathsf{dg}}Coh^G(\tilde{Y}^n)$. 
Working with the derived category of $G$-equivariant coherent sheaves requires some modifications to our general setup (in particular to Conjecture \ref{conj:B}). 
One can circumvent this issue to some extent, and work on the smooth scheme $\tilde{Y}^n$ rather than on $Y^n$, but let us ignore this point and act as if $Y^n$ itself were smooth.

Theorem \ref{thm:hmscy} raises three natural questions:
\begin{enumerate}
\item \label{it:mir} Is the mirror map $\psi$ that appears in Theorem \ref{thm:hmscy} the same as that appearing in Hodge-theoretic mirror symmetry, which is defined in terms of solutions to the Picard-Fuchs equation? In particular, the mirror map $\psi^*$ that appears in Theorem \ref{thm:hmscy} was not determined in \cite{Sheridan2015}, beyond the first term (see Remark \ref{rmk:waffle} for more on this). 
\item \label{it:odgems} Does Theorem \ref{thm:hmscy} imply Hodge-theoretic or enumerative mirror symmetry?
\item \label{it:vol} The Fukaya category $\EuF(X^n)$ comes with a natural Calabi-Yau structure: the simplest manifestation of this is the Poincar\'{e} duality pairing on Floer cohomology:
\[ \mathrm{Hom}^\bullet(K,L) \cong \mathrm{Hom}^{n-\bullet}(L,K)^\vee.\]
Under homological mirror symmetry, this corresponds to a Calabi-Yau structure on $\dbdg{Y^n}$: these are in one-to-one correspondence with relative holomorphic volume forms, i.e., non-vanishing sections of the canonical bundle $\Omega \in H^0(K_{Y})$ (in particular, Poincar\'{e} duality should correspond to the isomorphism
\[ \mathrm{Ext}^\bullet(\mathcal{E},\mathcal{F}) \cong \mathrm{Ext}^{n-\bullet}(\mathcal{F},\mathcal{E} \otimes K_Y)^\vee \cong \mathrm{Ext}^{n-\bullet}(\mathcal{F},\mathcal{E})^\vee,\]
where the first isomorphism is Serre duality, and the second is given by the isomorphism $\mathcal{O}_Y \cong K_Y$ corresponding to $\Omega$). 
Of course, we have infinitely many possible choices for the holomorphic volume form $\Omega$: any choice can be multiplied by a non-zero element of $\BbK_B$. 
This raises the final question: to which volume form $\Omega$ does the natural Calabi-Yau structure on $\EuF(X)$ correspond, under homological mirror symmetry?
\end{enumerate}

In light of our results (which, we recall, rely on Conjecture \ref{conj:B}, and modifications to deal with $G$-equivariant coherent sheaves), we can answer these questions: the answers to \eqref{it:mir} and \eqref{it:odgems} are `yes', and the answer to \eqref{it:vol} is `the Hodge-theoretically normalized volume form' -- see \S  \ref{sec:cystructures}, particularly Theorem \ref{thm:cystructures} for details on the latter point.

In particular, because $Y^n$ is maximally unipotent, we can apply Theorem \ref{thm:main} to the homological mirror symmetry quasi-equivalence of Theorem \ref{thm:hmscy}: so we obtain a new proof of Hodge-theoretic mirror symmetry for the mirror pairs $(X^n,Y^n)$. 
By Theorem \ref{thm:hodgemsinfo}, this allows us to compute the matrix of quantum cup product with $[\omega]$ on $X^n$, which of course contains information about certain three-point genus-zero Gromov-Witten invariants of $X^n$. 
In particular, in the case of the quintic $X^5$, we obtain a new proof that the curve counts predicted in \cite{Candelas1991} are correct. 
We remark that the complex scalar $c_1$ alluded to in \S\ref{subsec:hodgems} is normalized up to sign by the computation of the leading-order term of the mirror map in Theorem \ref{thm:hmscy}: so rather than having to input the leading-order term $2875$ in the Yukawa coupling in order to normalize $c_1$, we only need to input its sign $+1$.

\begin{rmk}
\label{rmk:waffle}
Let us make a philosophical remark about the proof of Hodge-theoretic mirror symmetry for the quintic that we just outlined. 
The previous proofs went by computing Gromov-Witten invariants directly, for example by equivariant localization. 
When proving directly that two \vhs{} are isomorphic, one could imagine that they match up to order one million, but fail to match up to order one million and one\footnote{We thank Rahul Pandharipande for pointing this out to us. One can see this, for example, from Proposition \ref{prop:cateq}: the matrix $A(q)$ that appears as part of an object of $\EuD_n$ could be altered at arbitrarily high order in $q$, and the corresponding object of $\EuC_n$ is still a perfectly good \vhs{}.}. 
Thus one needs to keep track of curve counts of all orders. 
The proof via homological mirror symmetry is of a different nature. 
Namely, one first proves that the categories match up to zeroth order in $Q$, then one proves that they match up to first order in $Q$ (although each comparison involves infinitely many $A_\infty$ structure maps, one only needs to compute a finite number of them to determine the structure up to $A_\infty$ quasi-equivalence). 
Then, one uses the fact that there is a one-dimensional `moduli space of $\ainf$ structures on the category'. 
This means we have matched up the origins in the respective moduli spaces of $A_\infty$ structures (i.e., the zeroth order categories), and we have also matched up the directions in which both categories are deforming (the first order categories). 
It then follows by the inverse function theorem that the families of categories are related by some formal diffeomorphism, which is the mirror map. 
This mirror map may appear to be undetermined: however, because homological mirror symmetry implies Hodge-theoretic mirror symmetry, this mirror map is uniquely determined by the fact that it must match up the canonical coordinates on both sides.
\end{rmk}

\begin{rmk}
    The version of Hodge-theoretic mirror symmetry that we extract from homological mirror symmetry is not the optimal result: one would ultimately hope to prove an isomorphism between the big $A$-model \vhs{} and the big $B$-model \vhs{}, which would imply an isomorphism of the associated Frobenius manifolds (see \cite{Barannikov1998}).
That should also be possible by extending the techniques presented in this paper to include `bulk deformations', although we have not carried that out. 
The key point is that HMS implies that $\EuC\EuO$ is an isomorphism (see Theorem \ref{thm:ociso2}): and $\EuC\EuO$ extends to an $L_\infty$ morphism, so the universal family of deformations of $\EuF(X)$ gets identified with the bulk deformed Fukaya category.
\end{rmk}

\subsection{Acknowledgments}

The authors are grateful to David Ben-Zvi, Chris Brav, Tom Coates, Rahul Pandharipande, Yongbin Ruan and Paul Seidel for very helpful
conversations. S.G. would like to thank the Institut Mittag-Leffler for
hospitality while working on this project.  N.S. is grateful to the Instituto
Superior T\'{e}cnico and to ETH Z\"{u}rich for hospitality while working on
this project.

\section{Variations of semi-infinite Hodge structures over formal punctured discs}
\label{sec:vhs}

In this section, we review the definition of a variation of semi-infinite Hodge structures (which we abbreviate `\vhs'), following Barannikov \cite{Barannikov2001}. 

\subsection{Definitions}
\label{subsec:vhsdef}

Variations of semi-infinite Hodge structures were introduced in \cite{Barannikov2001}. 
We recall a particular case of the definition here, following \cite[\S  2]{Sheridan2015a}. 
In this paper we will only consider \vhs{} over formal punctured discs, which we now define:

\begin{defn} 
\label{defn:coord}
Let $R$ be a complete discrete valuation ring with maximal ideal $\mathfrak{m}$, residue field $R/\mathfrak{m} = \C$, field of fractions $\BbK$, and valuation $v: \BbK^\times \to \Z$. 
We denote $\cM := \spec \BbK$, and we call such $\cM$ a \emph{formal punctured disc}. 
We denote $\mathcal{O}_\cM := \BbK$ and $T\cM := \deriv_\C \BbK$.
\end{defn}

\begin{defn}
A \emph{coordinate} on $\cM$ is an element $q \in R$ with $v(q) = 1$ (also known as a `uniformiser'). 
A choice of coordinate $q$ determines an isomorphism $R \cong \C\power{q}$, and similarly $\BbK \cong \C \laurents{q}$: i.e., an isomorphism of $\cM$ with the standard formal punctured disc. 
A coordinate also determines an isomorphism $T\cM \cong \BbK \cdot \partial_q$. 
\end{defn}

We define $\BbK \power{u}$ to be the completion of $\BbK[u]$ in the category of graded algebras, where $u$ has degree $2$. 
Note that the completion has no effect: $\BbK \power{u} \cong \BbK[u]$. 
Nevertheless we continue to use the notation $\BbK\power{u}$, as it reminds us that any graded $\BbK \power{u}$-module will always be completed with respect to a filtration by powers of $u$ in the category of graded modules, by convention (compare \cite[\S  3.6]{Sheridan2015}). 
Similarly, we denote the graded ring of formal Laurent series in $u$ by $\BbK \laurents{u}\cong \BbK[u,u^{-1}]$.
For any $f \in \BbK \power{u}$ or $\BbK \laurents{u}$, we denote
\[ f^\star(u) := f(-u).\]

\begin{defn}
\label{defn:prevhs}
Let $\cM := \spec \BbK$ be a formal punctured disc.
A \emph{$\Z$-graded unpolarized pre-\vhs} over $\cM$ is a pair $\EuH := (\EuE,\nabla)$, where:
\begin{itemize}
\item $\EuE$ is a graded $\BbK \power{u}$-module.
\item $\nabla$ is a flat connection\footnote{More precisely, there is a map $u\nabla: T\cM \otimes_\C \EuE \to \EuE$, such that $u\nabla_X s$ is $\BbK$-linear in $X$, additive in $s$, satisfies the Leibniz rule
\[ u\nabla_X(f \cdot s) = uX(f) \cdot s + f \cdot u\nabla_X s\]
for $f \in \BbK \power{u}$, and 
\[ [u\nabla_X, u\nabla_Y] = u^2\nabla_{[X,Y]}\]
for all $X,Y \in T\cM$.} 
\[ \nabla: T\cM \otimes \EuE \to u^{-1} \EuE, \]
of degree $0$.
\end{itemize}
\end{defn}

\begin{defn}
\label{defn:pol}
A \emph{polarization} for a pre-\vhs{} $\EuH=(\EuE,\nabla)$ is a pairing
\[ (\cdot,\cdot):\EuE \times \EuE  \to \BbK \power{u}\]
of degree $0$, satisfying the following conditions:
\begin{itemize}
\item $(\cdot,\cdot)$ is {\em sesquilinear}, i.e., it is additive in both inputs and
\[(f\cdot s_1,s_2) = (s_1,f^\star \cdot s_2) = f\cdot (s_1,s_2)\]
 for $f \in \BbK \power{u}$.
\item $(\cdot,\cdot)$ is covariantly constant with respect to $\nabla$, i.e. 
\[X(s_1,s_2) = (\nabla_X s_1,s_2) + (s_1,\nabla_X s_2).\]
\item The pairing is graded symmetric: precisely, there exists $n \in \Z/2$ (called the `dimension') such that
\[(s_1,s_2) = (-1)^{n+\mathsf{deg}(s_j)} (s_2,s_1)^\star,\]
(we observe that the pairing vanishes unless $\mathsf{deg}(s_1) = -\mathsf{deg}(s_2)$ by definition, hence there is no ambiguity in the choice of $j$ in the exponent). 
\end{itemize}
\end{defn}

\begin{defn}
\label{defn:vhs}
An unpolarized \vhs{} is an unpolarized pre-\vhs{} such that the $\BbK\power{u}$-module $\EuE$ is finitely-generated and free. 
\end{defn}

\begin{defn}
\label{defn:polvhs}
A polarization for a \vhs{} is a polarization for the underlying pre-\vhs{}, with the additional property that the pairing of $\BbK$-modules
\[ \EuE /u\EuE  \otimes_\BbK \EuE /u\EuE  \to \BbK\]
induced by $(\cdot,\cdot)$ is non-degenerate.
\end{defn}

\begin{lem}
\label{lem:realvhs}
Let $\cM$ be a formal punctured disc.
Then a $\Z$-graded unpolarized \vhs{} $\EuH = (\EuE,\nabla)$ over $\cM$ is equivalent to the following data:
\begin{itemize} 
\item A free, finite-rank, $\Z/2$-graded $\BbK$-module $\EuV \cong \EuV_{ev} \oplus \EuV_{odd}$.
\item A flat connection $\nabla$ on each $\EuV_\sigma$.
\item Decreasing filtrations 
\[ \ldots \supset \EuF^{\ge p} \EuV_{ev} \supset \EuF^{\ge p+1} \EuV_{ev} \supset \ldots\]
and
\[ \ldots \supset \EuF^{\ge p-\frac{1}{2}} \EuV_{odd} \supset \EuF^{\ge p + \frac{1}{2}} \EuV_{odd} \supset \ldots \]
which are called the \emph{Hodge filtrations}, and satisfy \emph{Griffiths transversality}:
\[ \nabla_v \EuF^{\ge p} \subset \EuF^{\ge p-1}.\]
\end{itemize}
An $n$-dimensional polarization on $\EuH$ is equivalent to covariantly constant bilinear pairings 
\[ (\cdot,\cdot):\EuV_\sigma \otimes \EuV_\sigma \to \BbK\]
for both $\sigma \in \Z/2$, such that 
\[ (\alpha,\beta) = (-1)^n (\beta,\alpha),\]
and with the property that
\[ (\EuF^{\ge p} \EuV_\sigma,\EuF^{\ge q} \EuV_\sigma) = 0 \mbox{ if $p+q > 0$}, \]
and the induced pairing
\[ (\cdot,\cdot): \mathsf{Gr}^p_{\EuF} \EuV_\sigma \otimes \mathsf{Gr}^{-p}_{\EuF} \EuV_\sigma \to \BbK\]
is non-degenerate, for all $p$.
\end{lem}
\begin{proof}
We give the construction in one direction (the reverse construction is clear). 
Let $\EuH$ be a $\Z$-graded \vhs, so $\widetilde{\EuE} := \EuE \otimes_{\BbK \power{u} } \BbK \laurents{u}$ is a free $\Z$-graded $\BbK[u,u^{-1}]$-module.
We have the \emph{periodicity isomorphisms} 
\[ \widetilde{\EuE}_k \overset{u \cdot}{\to} \widetilde{\EuE}_{k+2},\]
 so we can define 
\[ \EuV_{[k]} := \widetilde{\EuE}_k,\]
where different choices of $k$ mod $2$ are identified via the periodicity isomorphisms. 
Observe that the connection has degree $0$ and is $u$-linear, hence it descends to a connection on $\EuV$.

We define the Hodge filtrations by
\[ \EuF^{\ge p-\frac{k}{2}} \EuV_{[k]} :=\left( u^{\ge p} \cdot \EuE \right)_k \subset \widetilde{\EuE}_k. \]
It is easy to check that this respects the periodicity isomorphisms, hence is well-defined. 
Griffiths transversality follows from the fact that $\nabla$ maps $\EuE \to u^{-1} \EuE$.

We define the pairing of $\alpha, \beta \in \EuV_\sigma$ by choosing representatives $\tilde{\alpha} \in \widetilde{\EuE}_k$, $\tilde{\beta} \in \widetilde{\EuE}_{-k}$, then setting
\[ (\alpha,\beta)_\EuV := \iii^{-k} \left(\tilde{\alpha},\tilde{\beta}\right)_\EuE,\]
where $\iii := \sqrt{-1}$.
Observe that the degree assumptions ensure that the output lies in the degree-$0$ part of $\BbK \laurents{u}$, which is $\BbK$. 
The prefactor ensures that the pairing respects the periodicity isomorphisms, by sesquilinearity of the pairing on $\widetilde{\EuE}$, and also that it has the appropriate symmetry property, by symmetry of the pairing on $\widetilde{\EuE}$.

The pairing $(\cdot,\cdot)_\EuV$ is covariantly constant, by the corresponding property for $(\cdot,\cdot)_\EuE$. 
If $\alpha \in \EuF^{\ge p}$ and $\beta \in \EuF^{\ge q}$, then $\tilde{\alpha} \in u^{\ge p+ k/2} \EuE_k$ and $\tilde{\beta} \in u^{\ge q - k/2} \EuE_{-k}$, so their pairing lies in $u^{\ge p+q} \cdot \BbK\power{u}$. 
In particular, if $p+q > 0$ then the constant coefficient vanishes, so $(\alpha,\beta)_\EuV = 0$. 

We observe that there is a natural isomorphism 
\[ (\EuE/u\EuE)_k \cong \mathsf{Gr}_\EuF^{ -\frac{k}{2}} \EuV_{[k]}.\]
Therefore, the non-degeneracy property of $(\cdot,\cdot)_\EuV$ follows from that of $(\cdot,\cdot)_\EuE$.
\end{proof}

\begin{rmk}
\label{rmk:shiftV}
It is more standard to allow the pairing $(\cdot,\cdot)$ to have a non-zero degree, and to consider shifts of the grading. 
We prefer to shift whatever \vhs{} we are considering, so that the pairing has degree $0$ (the higher residue pairing always has degree $0$ with respect to the standard grading on cyclic homology).
\end{rmk}

\begin{defn}
\label{defn:oppfilt}
Given a $\Z$-graded \vhs{} over a formal punctured disc, an \emph{opposite filtration} (or a \emph{splitting for the Hodge filtration}) is a pair of increasing filtrations
\[ \ldots \subset \EuW_{\le p} \EuV_{ev} \subset \EuW_{\le p+1} \EuV_{ev} \subset \ldots\]
and
\[ \ldots \subset \EuW_{\le p-\frac{1}{2}} \EuV_{odd} \subset \EuW_{\le p + \frac{1}{2}} \EuV_{odd} \subset \ldots \]
preserved by $\nabla$, and such that the inclusion maps induce isomorphisms:
\begin{equation}
\label{eqn:FWopp}
 \EuF^{\ge p}\EuV_\sigma \oplus \EuW_{\le p-1} \EuV_\sigma \overset{\sim}{\to} \EuV_\sigma
\end{equation}
for all $p \in \Z + \frac{\sigma}{2}$. 
\end{defn}

An opposite filtration determines isomorphisms
\begin{equation}
\label{eqn:splitting}
\EuV_\sigma \cong \bigoplus_p \EuV_{\sigma}^{(p)} \cong  \mathsf{Gr}_{\EuF} \EuV_\sigma \cong \mathsf{Gr}^{\EuW} \EuV_\sigma,
\end{equation}
where
\[ \EuV^{(p)}_\sigma := \EuF^{\ge p} \cap \EuW_{\le p} \EuV_\sigma,\]
and the isomorphisms are induced by the inclusions $\EuV_\sigma^{(p)} \hookrightarrow \EuV_\sigma$.

\subsection{Monodromy weight filtration}
\label{subsec:monwf}

Let $\EuH$ be a $\Z$-graded polarized \vhs{} over a formal punctured disc $\cM$, which is equivalent to the data $(\EuV,\EuF^{\ge \bullet},\nabla,(\cdot,\cdot))$ described in Lemma \ref{lem:realvhs}. 

Assume that $(\EuV,\nabla)$ has a regular singular point at $q=0$, whose monodromy $T$ is unipotent of order $n$: $(T-I)^{n+1} = 0$.
Let $\widetilde{\EuV} \subset \EuV$ denote the Deligne lattice (i.e., the canonical extension over $0$, a free $R$-module, where $R \cong \C\power{q} \subset \BbK$; see e.g. \cite[\S  II.2.e]{Sabbah2007}), $\widetilde{\EuV}_0:= \widetilde{\EuV}/q\widetilde{\EuV}$ (the fibre at $0$ of the canonical extension, a $\C$-vector space), and define the associated monodromy weight filtrations 
\[ 0 \subset MW_{\le -n}\, \widetilde{\EuV}_0 \subset \ldots \subset MW_{\le n} \, \widetilde{\EuV}_0 = \widetilde{\EuV}_0\]
(using the nilpotent endomorphism which is the residue of the connection) and similarly $\EuM\EuW_{\le p} \EuV$ (using the nilpotent endomorphism which is the log monodromy). 
We define the increasing filtration
\begin{equation}
\label{eqn:monweightk}
 \EuW_{\le p} \EuV_{\sigma} := \EuM\EuW_{\le 2p} \EuV_\sigma,
\end{equation}
where $p \in \Z + \frac{\sigma}{2}$ (as in Definition \ref{defn:oppfilt}); similarly we define the filtration $W_{\le p} \, \widetilde{\EuV}_0 := MW_{\le 2p} \widetilde{\EuV}_0$.

Suppose that the filtration $\EuW_{\le p}$ splits the Hodge filtration $\EuF^{\ge p}$, in the sense of \eqref{eqn:FWopp}, and suppose furthermore that the splitting \emph{extends over $0$}: i.e., if we define
\[ \widetilde{\EuV}^{(p)} := \EuV^{(p)} \cap \widetilde{\EuV},\]
then the direct sum of inclusion maps
\[ \bigoplus_p \widetilde{\EuV}^{(p)} \to \widetilde{\EuV}\]
induces an isomorphism. 
Setting $q=0$, we then obtain an isomorphism
\begin{equation}
\label{eqn:v0split}
 \bigoplus_p \widetilde{\EuV}_0^{(p)} \cong \widetilde{\EuV}_0
\end{equation}
(one can say `the limiting Hodge filtration $F^{\ge \bullet}_{\mathrm{lim}}$ splits the weight filtration $W_{\le \bullet}$ on $\widetilde{\EuV}_0$'; the $B$-model \vhs{} we consider will extend over $0$ by Schmid's nilpotent orbit theorem \cite{Schmid1973}).
 
The connection $\nabla$ respects the filtration $\EuW_{\le \bullet}$, and therefore induces a connection $\nabla^\EuW$ on $\mathsf{Gr}^\EuW \EuV$; this connection is trivial, and its flat sections are canonically identified with $\mathsf{Gr}^W \widetilde{\EuV}_0$. 
Thus, we have a canonical isomorphism
\[ \mathsf{Gr}^W \, \widetilde{\EuV}_0 \otimes_\C \BbK \cong \mathsf{Gr}^\EuW \EuV. \]
Using the splittings, this gives an isomorphism
\begin{equation}
\label{eqn:Wflatbas}
 \widetilde{\EuV}_0 \otimes_\C \BbK \cong \EuV,
\end{equation}
which identifies $\widetilde{\EuV}_0^{(p)} \otimes_\C \BbK$ with $\EuV^{(p)}$.

Suppose, furthermore, that the flat sections of $\nabla^\EuW$ are contained in $\widetilde{\EuV}^{(p)} \subset \EuV^{(p)} \cong \mathsf{Gr}^\EuW_p \EuV$; then the isomorphism \eqref{eqn:Wflatbas} identifies $\widetilde{\EuV}_0 \otimes R$ with $\widetilde{\EuV}$.

\begin{lem}
\label{lem:canconn}
(Compare \cite[Theorem 11]{Deligne1997})
Suppose that our \vhs{} is above. 
Then given a choice of coordinate $q \in \BbK$, if we write the connection $\nabla$ in the trivialization \eqref{eqn:Wflatbas}, it takes the form
\[ \nabla_{q \partial_q} = q \partial_q + A(q),\]
for some $A(q) \in \End_\C\left(\widetilde{\EuV}_0\right)_{-1} \otimes \C\power{q}$. 
The subscript `$-1$' denotes the subspace of endomorphisms of $\widetilde{\EuV}_0$ that have degree $-1$ with respect to the $\frac{1}{2}\Z$-grading \eqref{eqn:v0split}.
\end{lem}
\begin{proof} 
Because \eqref{eqn:Wflatbas} identifies $\widetilde{\EuV}_0 \otimes R$ with the canonical extension, which is a logarithmic lattice, $A(q) \in \End(\widetilde{\EuV}_0) \otimes \BbK$ has no poles at $q=0$.
Because $\widetilde{\EuV}_0$ gets identified with the flat sections of $\mathsf{Gr}^\EuW$, $A(q)$ sends $\EuW_{\le p} \to \EuW_{\le p-1}$. 
Furthermore, by Griffiths transversality, $A(q)$ sends $\EuF^{\le p} \to \EuF^{\le p-1}$. 
Therefore, $A(q)$ maps $\EuV^{(p)} \to \EuV^{(p-1)}$, so its Taylor coefficients have degree $-1$ as claimed.
\end{proof}

\subsection{The pairing}
\label{subsec:normpair}

Suppose that $\EuH$ is as in the previous section, and suppose furthermore that the pairing \emph{extends over $q=0$}: i.e, when we restrict the pairing $(\cdot,\cdot)_\EuV$ to $\widetilde{\EuV}$, it takes values in $R \subset \BbK$, and is a non-degenerate pairing of free $R$-modules (in the language of \cite[III.1.12]{Sabbah2007}, the pairing has \emph{weight $0$}).
Evaluating at $q=0$ then defines a non-degenerate $\C$-bilinear pairing
\begin{equation}
\label{eqn:pair0}
 (\cdot,\cdot)_0: \widetilde{\EuV}_0 \otimes \widetilde{\EuV}_0 \to \C.
\end{equation}

\begin{lem}
\label{lem:pairfix}
The $\BbK$-bilinear pairing $(\cdot,\cdot)$ on $\EuV$ is uniquely determined by $(\cdot,\cdot)_0$. 
Furthermore, the residue of $\nabla$  is skew-adjoint with respect to $(\cdot,\cdot)_0$ (in the setting of Lemma \ref{lem:canconn}, the residue is equal to $A(0)$).
\end{lem}
\begin{proof}
Choose a basis for $\widetilde{\EuV}_0$: it determines a basis for $\widetilde{\EuV}$ via \eqref{eqn:Wflatbas}. 
Let $M(q) \in \mathsf{Mat}_{d \times d}(R)$ be the matrix for the pairing, with respect to this basis. 
Because the pairing is covariantly constant, we have
\[q \partial_q M(q) = A(q)^t \cdot M(q) + M(q) \cdot A(q),\]
where $A(q)$ is the matrix from Lemma \ref{lem:canconn}. 
We expand this equation in powers of $q$: the $q^0$ term says that 
\[ A(0)^t \cdot M(0) + M(0) \cdot A(0) = 0,\]
which precisely means that the residue $A(0)$ is skew-adjoint with respect to $(\cdot,\cdot)_0$.

Now we show that, given $M(0)$, we can solve inductively for the higher terms in the Taylor expansion of $M(q)$.
If $M(q) = \sum_{k \ge 0} M_k q^k$, then the $q^k$ term of the equation says that
\[ k M_k = A(0)^t \cdot M_k + M_k \cdot A(0) + \Phi_k(A(q),M_0,\ldots,M_{k-1}).\]
Because $A(0)$ and $A(0)^t$ are nilpotent, their only eigenvalues are $0$: so $k \cdot \id - A(0)$ and $A(0)^t$ have no common eigenvalues, as $k>0$. 
Therefore, by \cite[Lemma 2.16]{Sabbah2007}, the equation can be solved uniquely for $M_k$. 
By induction, all $M_k$ are determined uniquely by $M_0$, as required.
\end{proof}

\begin{lem}
\label{lem:pairgrad}
The pairing 
\[(\cdot,\cdot)_0: \widetilde{\EuV}_0^{(p)} \otimes \widetilde{\EuV}_0^{(q)} \to \C\]
is non-degenerate, if $p+q = 0$, and vanishes otherwise.
\end{lem}
\begin{proof}
It follows from Lemma \ref{lem:realvhs} that the pairing vanishes for $p+q > 0$, and that it is non-degenerate for $p+q = 0$. 
To show that it also vanishes for $p+q<0$, we use the fact that $(W_{\le p}, W_{\le q})_0 = 0$ for $p+q<0$: this is a standard consequence of the fact that the monodromy weight filtration is constructed using the nilpotent endomorphism given by the residue of the connection, which is skew-adjoint with respect to $(\cdot,\cdot)_0$ by Lemma \ref{lem:pairfix} (see e.g. \cite[Lemma 6.4]{Schmid1973}). 
\end{proof}

\begin{defn}
\label{defn:hodgetate}
Let $\EuH$ be a $\Z$-graded polarized \vhs{} over a formal punctured disc.
We say that $\EuH$ is \emph{Hodge-Tate} if:
\begin{enumerate}
\item $\nabla$ has a regular singular point at $0$ with unipotent monodromy of order $n$;
\item the induced weight filtration $\EuW_{\le \bullet}$ splits the Hodge filtration, and the splitting extends over $q=0$;
\item the flat sections of $\nabla^\EuW$ are contained in $\widetilde{\EuV}^{(p)}$;
\item the pairing extends over $q=0$.
\end{enumerate}
In other words, these are precisely the conditions we need in order to apply Lemmas \ref{lem:canconn}, \ref{lem:pairfix} and \ref{lem:pairgrad}.
\end{defn}

\subsection{Volume forms}
\label{subsec:normvol}

\begin{defn}
\label{defn:volform}
Suppose that $\EuH$ is Hodge-Tate, and $\widetilde{\EuV}_0^{(n/2)}$ is one-dimensional. 
It follows that $\EuF^{\ge n/2}\EuV$ is $1$-dimensional: we call an element $\Omega \in \EuF^{\ge n/2} \EuV$ a \emph{volume form}.
\end{defn}

\begin{defn}
\label{defn:normOm} 
Observe that \eqref{eqn:Wflatbas} identifies
\[  \widetilde{\EuV}_0^{(n/2)} \otimes_\C \BbK \cong \EuF^{\ge n/2} \EuV .\] 
We say that a volume form $\Omega$ is \emph{normalized} if, under this isomorphism, it corresponds to a constant element, i.e., an element of $\widetilde{\EuV}_0^{(n/2)}\otimes \C$.
\end{defn}

\begin{rmk}
If $\Omega$ is a normalized volume form, then $[\Omega] \in \mathsf{Gr}_{n/2}^{\EuW} \EuV$ is called the \emph{dilaton shift} (see \cite[\S  2.2.2]{Coates2009}). 
The terminology `normalized volume form' comes from \cite{coxkatz}, see in particular \cite[Proposition 5.6.1]{coxkatz}.
\end{rmk}

\subsection{Canonical coordinates}
\label{subsec:cancoord}

Suppose that the conditions in Definition \ref{defn:normOm} are satisfied, and $\Omega \in \EuF^{\ge n/2} \EuV$ is a normalized volume form. 
By the definition of being normalized, $[\Omega] \in \mathsf{Gr}_{n/2}^{\EuW}$ is flat;
it follows that there is a well-defined map
\begin{eqnarray}
\label{eqn:KS}
KS: T\cM & \to & \mathsf{Gr}^\EuW_{n/2-1} \EuV,\\
KS(v) & := & [\nabla_v \Omega].
\end{eqnarray}
This is called the \emph{Kodaira-Spencer map}.

\begin{rmk}
A \vhs{} is said to be \emph{miniversal} if the (analogue of the) Kodaira-Spencer map is an isomorphism onto all of $\mathsf{Gr}^{\EuW}$ (compare \cite[Definition 2.8]{Coates2009}). 
However we are only considering the case of a one-dimensional base here, with trivial grading (small quantum cohomology as opposed to big quantum cohomology), so the most we could hope for is that \eqref{eqn:KS} is an isomorphism (compare \cite[Remark 2.13]{Coates2009}).
\end{rmk}

\begin{defn}
\label{defn:cancoord}
Observe that the isomorphism \eqref{eqn:Wflatbas} identifies
\[ \widetilde{\EuV}_0^{(n/2-1)} \otimes \BbK \cong \mathsf{Gr}^\EuW_{n/2-1} \EuV .\]
We call a coordinate $q \in \BbK$ a \emph{canonical coordinate} if $KS(q \partial_q)$ is constant under this identification, i.e., lies in $\widetilde{\EuV}_0^{(n/2-1)} \otimes \C$. 
\end{defn}

\begin{rmk}
Equivalently, $q$ is a canonical coordinate if the coefficient of the matrix $A(q)$ (from Lemma \ref{lem:canconn}) that sends $\widetilde{\EuV}_0^{(n/2)} \to \widetilde{\EuV}^{(n/2-1)}_0$ is constant.
\end{rmk}

\begin{lem}
\label{lem:cancoordun}
Suppose that $\EuH$ is Hodge-Tate, and $\widetilde{\EuV}_0^{(n/2)}$ is one-dimensional. 
Then if a canonical coordinate $q$ exists, it is unique up to multiplication by a non-zero complex scalar. 
This scalar is uniquely determined if we specify a non-zero cotangent vector $\alpha \in \Omega^1_0 \cM := \mathfrak{m}/\mathfrak{m}^2$, and required that $dq=\alpha$ at $0$.

If, furthermore, $\widetilde{\EuV}_0^{(n/2-1)}$ is one-dimensional, then a canonical coordinate $q$ necessarily exists.
\end{lem}
\begin{proof}
Let $e_{n/2}$ span $\widetilde{\EuV}_0^{(n/2)}$. 
Observe that $e_{n/2}$ represents a normalized volume form. 
In the setting of Lemma \ref{lem:canconn}, we have
\[ KS(q \partial_q) = \nabla_{q \partial_q} e_{n/2} = A(q) \cdot e_{n/2};\]
so $q$ is a canonical coordinate if and only if $A(q) \cdot e_{n/2}$ is constant. 
Observe that, because $A(0)$ is the matrix for the residue of the connection $\nabla$, which induces the monodromy weight filtration on $\widetilde{\EuV}_0$, we must have $A(0)\cdot e_{n/2} \neq 0$.

Suppose that $q$ is a canonical coordinate, and let $\tilde{q}$ be another. 
We have  $\tilde{q} = f(q) \cdot q$ for some $f(q) \in \C\power{q}$ with $f(0) \neq 0$. 
We then have
\[ q \partial_q = \left( 1 + q\cdot \frac{f'(q)}{f(q)} \right) \cdot \tilde{q} \partial_{\tilde{q}}.\]
As a consequence,
\[ [\nabla_{\tilde{q} \partial_{\tilde{q}}} e_{n/2}] = \frac{1}{1+q \cdot \frac{f'(q)}{f(q)}} \cdot A(q) \cdot e_{n/2}.\]
Therefore, $\tilde{q}$ is a canonical coordinate if and only if
\[ \frac{1}{1+q \cdot \frac{f'(q)}{f(q)}} = C\]
(since $A(q) \cdot e_{n/2} \neq 0$). 
Given the assumption $f(0) \neq 0$, the only solutions to this equation are the constants, $f(q) = c$.  
If we require that $d \tilde{q} = \alpha$ at $0$, then $c$ is uniquely determined. 

We leave the existence statement as an exercise.
\end{proof}

\subsection{Normal form}
\label{subsec:normf}

The results of the preceding sections show that \vhs{} which are Hodge-Tate, and such that canonical coordinates can be defined, can be put in a nice normal form. 
This `normal form' statement can efficiently be summarized as an equivalence of categories, in the same style as the Riemann-Hilbert correspondence. 
In this section we state this result precisely.
First we define the categories involved.

\begin{defn}
\label{defn:vhscat}
We define a category $\EuC_n$. 
Objects of $\EuC_n$ consist of:
\begin{itemize}
\item A formal punctured disc $\cM$ (i.e., a field $\BbK$ as in Definition \ref{defn:coord}).
\item A non-zero $\alpha \in \Omega^1_0 \cM$.
\item A $\Z$-graded polarized \vhs{} $\EuH$ over $\cM$ which is Hodge-Tate in the sense of Definition \ref{defn:hodgetate}, such that $\widetilde{\EuV}_0^{(n/2)}$ is $1$-dimensional, and which admits a canonical coordinate.
\end{itemize}
A morphism in $\EuC_n$ consists of an isomorphism 
\[ \psi: \cM_1 \to \cM_2\]
such that $\psi^* \alpha_2 = \alpha_1$,
and an isomorphism of \vhs:
\[ \phi: \psi^* \EuH_2 \to \EuH_1.\]
\end{defn}

\begin{defn}
\label{defn:vecspacecat}
We define the category $\EuD_n$ whose objects consist of:
\begin{itemize}
\item A finite-dimensional $\Z$-graded $\C$-vector space $V$;
\item A non-degenerate bilinear pairing 
\[ \langle \cdot, \cdot \rangle: V^{\otimes 2} \to \C\]
on $V$, of degree $0$, and such that $\langle \alpha, \beta \rangle = (-1)^n \langle \beta, \alpha \rangle$;
\item An element $A(q) \in \End_\C(V)_{2} \otimes_\C \C\power{q}$, 
\end{itemize}
such that:
\begin{itemize}
\item The grading is concentrated in degrees between $-n$ and $n$;
\item For all $k$, the map
\[ A(0)^{k}: V_{-k} \to V_{k} \]
is an isomorphism;
\item $A(0)$ is self-adjoint with respect to the pairing;
\item $V_{-n}$ is one-dimensional;
\item The component of $A(q)$ mapping $V_{-n} \to V_{-n+2}$ is constant, i.e., lies in $\End_\C(V)_2$.
\end{itemize}
The morphisms in this category are isomorphisms of complex vector spaces, preserving the grading, pairing and $A(q)$.
\end{defn}

\begin{prop}
\label{prop:cateq}
Given an object of $\EuD_n$, we define an element of $\EuC_n$ as follows:
\begin{itemize}
\item The \vhs{} is over the standard formal punctured disc $\cM := \spec \BbK$, where $\BbK := \C\laurents{q}$.
\item $\EuE := V \otimes_\C \BbK\power{u}$, with the induced $\Z$-grading.
\item The connection is
\[ \nabla_{q \partial_q} (\alpha) := q \partial_q (\alpha) - u^{-1} A(q) \cdot \alpha,\]
extended $\C\power{u}$-linearly.
\item The pairing is defined in three steps: first, define $(\alpha,\beta)_0 := \langle \alpha, \beta \rangle$; then, extend $(\cdot,\cdot)_0$ to the unique $\BbK$-bilinear extension of the pairing on $V$ that is covariantly constant (see Lemma \ref{lem:pairfix}); finally, extend the pairing $\BbK\power{u}$-sesquilinearly.
\end{itemize}
This defines a functor from $\EuD_n$ to $\EuC_n$: this functor is an equivalence of categories.  
\end{prop}
\begin{proof}
That this is a functor is easily shown. 
The inverse functor was constructed in \S \S  \ref{subsec:monwf}--\ref{subsec:cancoord}. 
Namely, to an object of $\EuC_n$, we associate the vector space $V:=\widetilde{\EuV}_0$, with the $\Z$-grading 
\[ V_p := \widetilde{\EuV}_0^{(-2p)}\]
from \eqref{eqn:v0split}, the pairing $\langle \cdot, \cdot \rangle := (\cdot,\cdot)_0$, and the endomorphism $A(q)$ of Lemma \ref{lem:canconn}, where $q$ is the unique canonical coordinate so that $dq = \alpha$ at $0$ (Lemma \ref{lem:cancoordun}). 
The fact that $q$ is a canonical coordinate implies that the component of $A(q)$ mapping $V_{-n} \to V_{-n+2}$ is constant. 
These are mutually inverse functors, by construction.
\end{proof}

\section{Hodge-theoretic mirror symmetry}

\subsection{The $A$-model \vhs}
\label{subsec:QHnorm}

Let $X$ be as in \S \ref{subsec:setup}.

\begin{defn}
\label{defn:avhs}
We define the (small) $A$-model \vhs, $\EuH^A(X,\omega) := (\EuE,\nabla,(\cdot,\cdot))$ (compare, e.g., \cite[\S  2.4]{Coates2009}): 
\begin{eqnarray*}
\EuE& := & H^{\bullet}(X;\C) \otimes_\C \BbK_A\power{u} [n] \,\,\,\,\,\,\mbox{ (where the `$[n]$' denotes a degree shift)}\\
\nabla_{Q \partial_Q} \alpha & := & Q \partial_Q(\alpha) - u^{-1} [\omega] \star \alpha\\
(\alpha,\beta) & := & (-1)^{n(n+1)/2} \int_X \alpha \cup \beta^\star.
\end{eqnarray*}
It is a $\Z$-graded polarized \vhs{} over $\cM_A$. 
In the formula for the connection $\nabla$ (which is called the `quantum differential equation'), we recall that `$\star$' denotes the quantum cup product, defined by counting rational curves $u$: each curve is weighted by 
$Q^{\omega(u)} \in \BbK_A$. 
In the formula for the pairing $(\cdot,\cdot)$, we recall that `$\beta^\star$' denotes $\beta(-u)$.
\end{defn}

\begin{rmk}
Observe that we are only considering a single K\"{a}hler class (and its multiples), rather than the entire K\"{a}hler cone, so even calling this the `small $A$-model \vhs' is over-stating it. \end{rmk}

\begin{defn}
\label{defn:Dnob}
We define an object of the category $\EuD_n$ defined in Definition \ref{defn:vecspacecat}, by setting
\begin{itemize}
\item $V := H^\bullet(X;\C)[n]$;
\item The pairing on $V$ is the intersection pairing, together with a normalization factor: 
\[  \langle \alpha,\beta \rangle := (-1)^{n(n+1)/2} \iii^{|\beta|-n} \int_X \alpha \cup \beta;\]
\item $A(Q)$ is the endomorphism given by quantum cup product with the class $[\omega]$:
\[ A(Q) \cdot \alpha := [\omega] \star \alpha.\]
\end{itemize}
Observe that these meet the conditions required for an object of $\EuD_n$: in particular, $A(0) = [\omega] \cup$, so
\[ A(0)^k: H^{n-k}(X) \to H^{n+k}(X)\]
is an isomorphism, by Hard Lefschetz, and the map $A(Q): H^0(X;\C) \to H^2(X;\BbK_A)$ given by quantum cup product with $[\omega]$ is constant, because the identity in $H^0(X)$ is also an identity for the quantum cup product. 
Finally, we have
\[ \langle \alpha, \beta \rangle = (-1)^{n(n+1)/2+|\alpha|\cdot|\beta|} \iii^{|\beta| - n}  \int \beta \cup \alpha = (-1)^n \langle \beta,\alpha \rangle,\]
using $|\alpha|+|\beta| = 2n$.
\end{defn}

\begin{lem}
\label{lem:avhsDn}
By Proposition \ref{prop:cateq}, there is a unique object of $\EuC_n$ (up to isomorphism) corresponding to the object of $\EuD_n$ from Definition \ref{defn:Dnob}. 
It is isomorphic to the $A$-model \vhs, $\EuH^A(X)$.
\end{lem}
\begin{proof}
The only part of the proof that is not tautological is to check that the pairing $(\alpha,\beta) = \int_X \alpha \cup \beta^\star$ is the covariantly constant, $\BbK_A\power{u}$-sesquilinear extension of $\langle \cdot, \cdot \rangle$. 
It is clear that $(\cdot,\cdot)$ is sesquilinear, and one easily checks that it is covariantly constant, because $[\omega] \star$ is self-adjoint: hence it is the unique such extension.
\end{proof}

Observe that the normalized volume forms in the $A$-model \vhs{} are spanned by the identity $e \in H^0(X;\C)$, and the canonical coordinates are the complex multiples of the K\"{a}hler parameter $Q$.

\begin{rmk}
Lemma \ref{lem:avhsDn} can be interpreted as follows. 
Suppose we know the small $A$-model \vhs{} up to isomorphism, i.e., up to isomorphism in $\EuC_n$. 
How much \emph{information} about genus-zero Gromov-Witten invariants does this isomorphism class really contain? 
Lemma \ref{lem:avhsDn} gives us the answer: it contains the same information as the corresponding isomorphism class in $\EuD_n$. 
We will work through the example of hypersurfaces in projective space in \S \ref{subsec:hyps}.
\end{rmk}

\subsection{The $B$-model \vhs}
\label{subsec:Bvhs}

Let $Y \to \cM_B$ be as in \S \ref{subsec:setup}.

\begin{defn}
\label{defn:bvhs}
The (small) \emph{$B$-model \vhs}, $\EuH^B(Y)$, is a $\Z$-graded polarized \vhs{} over $\cM_B$. 
We define it by defining the corresponding data $(\EuV,F^{\ge \bullet},\nabla,(\cdot,\cdot))$ in accordance with Lemma \ref{lem:realvhs}:
\begin{itemize}
    \item $\EuV := H^\bullet_{\mathsf{dR}}(Y/\cM_B)$ is the relative de Rham cohomology of $Y$, with the $\Z$-grading collapsed to a $\Z/2$-grading. \item  The filtration $F^{\ge s}\EuV$ is a modification of the classical Hodge filtration:
\[ F^{\ge s} \EuV := \bigoplus_{p} H^p\left(\Omega^{\ge p+2s}_{Y/\cM_B}\right).\]
\item The connection $\nabla$ is the Gauss-Manin connection, see for instance \cite{Katz1968}.
\item The pairing is the intersection pairing:
\[ (\alpha,\beta) := \int_Y \alpha^\vee \wedge \beta,\]
where $\alpha^\vee := \iii^{|\alpha|} \alpha$ (compare the definition of the Mukai pairing in \cite{Caldararu2005}). 
\end{itemize}
One easily verifies that the pairing is covariantly constant and compatible with the Hodge filtration in the required way. 
One can also verify that the corresponding $\BbK_B\power{u}$-module is isomorphic to
\[ \EuE^B(Y) := \bigoplus_{i \in \Z}  u^i \cdot
F^{-i} H_{\mathsf{dR}}^{\bullet-2i}(Y/\cM_B) \]
(compare \S \ref{subsec:vhsBside}).
\end{defn}

Observe that
\[ F^{\ge \frac{n}{2}} \EuV \cong H^0\left(\Omega^n_{Y/\cM_B}\right).\]
Hence the terminology in Definition \ref{defn:volform}: a volume form in $\EuH^B(Y)$ is the same thing as a section $\Omega \in \Gamma(\Omega^n_{Y/\cM_B})$, i.e., a relative volume form on $Y \to \cM_B$. 

We recall the classical \emph{Kodaira--Spencer map},
\[ \mathsf{KS}: \mathcal{T}(\cM_B / \spec \C) \to H^1(Y,\mathcal{T}(Y/\cM_B)) \]
(see \cite[\S A.6]{Perutz2015} for the definition we use). 
The map
\[ [\nabla_v] : \mathsf{Gr}^s_\EuF \EuE^B(Y) \to \mathsf{Gr}^{s-1}_\EuF \EuE^B(Y)\]
induced by the connection is identified with the map
\[ \iota_{\mathsf{KS}(v)}: H^\bullet(\Omega^{-\bullet}) \to H^\bullet(\Omega^{-\bullet})\]
(compare \cite[Theorem 10.4]{Voisin2002a}).

\subsection{Mirror symmetry}
\label{subsec:closedms}

Let $X$ and $Y$ be as in \S \ref{subsec:setup}.

\begin{defn}[= Definition \ref{defn:hodgems}]
\label{defn:hodgems2}
We say that $X$ and $Y$ are \emph{Hodge-theoretically mirror} if there is an isomorphism of formal punctured discs
\[ \psi: \cM_A \to \cM_B\]
(called the \emph{mirror map}), and an isomorphism of \vhs,
\[ \EuH^A(X) \cong \psi^* \EuH^B(Y).\] 
\end{defn}

\begin{thm}
\label{thm:hodgemsinfo}
Suppose that $X$ and $Y$ are Hodge-theoretically mirror in the sense of Definition \ref{defn:hodgems2}. 
Then
\begin{itemize}
    \item The mirror map 
\[ \psi: \cM_A \to \cM_B\]
is uniquely determined up to multiplication by a complex scalar $c_1$ (see Lemma \ref{lem:cancoordun}).
\item $\EuH^B(Y)$ contains all information about the object of $\EuD_n$ given in Definition \ref{defn:Dnob}, up to substitution $A(Q) \mapsto A(Q/c_1)$.
\end{itemize} 
In particular, Hodge-theoretic mirror symmetry allows us to compute the $A(Q)$, the matrix of quantum cup product with $[\omega]$, from $\EuH^B(Y)$ (up to the ambiguity in $c_1$).
\end{thm}
\begin{proof}
We have seen that $\EuH^A(X)$ is an object of $\EuC_n$. 
If $X$ and $Y$ are Hodge-theoretically mirror, then $\EuH^B(Y)$ is also an object of $\EuC_n$, with $\alpha := dq(0) \in \Omega^1_{\cM_B}$.
We must have $c_1 \cdot dq(0) = (\psi^{-1})^*dQ(0)$ for some  $c_1 \in \C^*$, so if we equip $\EuH^A(X)$ with the coordinate $Q/c_1$ instead of $Q$, the resulting objects of $\EuC_n$ are isomorphic: then the corresponding objects of $\EuD_n$ are isomorphic by Proposition \ref{prop:cateq}. 
\end{proof}

\subsection{Application: hypersurfaces in projective space}
\label{subsec:hyps}

We recall the example from \S \ref{subsec:app}: $X^n$ is a degree-$n$ Fermat hypersurface in $\mathbb{CP}^{n-1}$, with integral symplectic form $\omega$, and $Y^n = \tilde{Y}^n/G$ is its mirror. 
As explained there, homological mirror symmetry \cite[Theorem 1.8]{Sheridan2015}, together with our main theorem (Theorem \ref{thm:main}) imply that they are also Hodge-theoretically mirror. 
The aim of this section is to answer the question: how much information about Gromov-Witten invariants of $X^n$ does this give us?

There is an action of the character group $G^*$ on $X^n$, and $\EuH^A(X^n)^{G^*}$ is precisely the Hodge part of the cohomology, i.e., the part generated by the K\"{a}hler class $[\omega]$ (see \cite[Lemma 7.5]{Sheridan2013} for a proof of this fact). 
This is the part that has interesting information about Gromov-Witten invariants, so this is the part we will focus on. 
There is also an action of $G^*$ on $H^\bullet_{\mathsf{dR}}(Y^n)$, and $\EuH^B(Y^n)^{G^*} \cong \EuH^B(\tilde{Y}^n)^G$.
The proof of homological mirror symmetry in \cite{Sheridan2015} makes it clear that mirror symmetry matches up these $G^*$-actions: so the resulting isomorphism of \vhs{} identifies
\[ \EuH^A(X^n)^{G^*}   \cong \psi^*\EuH^B(\tilde{Y}^n)^G.\]

Now, we recall that the leading term in the mirror map is determined in Theorem \ref{thm:hmscy}, up to sign. 
This means that in fact, $\EuH^A(X^n)^{G^*}$ and $\EuH^B(\tilde{Y}^n)^G$ are isomorphic as objects of $\EuC_n$ (up to the sign ambiguity): in particular, the ambiguity in $c_1$ from Theorem \ref{thm:hodgemsinfo} is removed, up to the sign. 
It follows that the corresponding objects of $\EuD_n$ are isomorphic (potentially up to the substitution $A(q) \mapsto A(-q)$), by Proposition \ref{prop:cateq}. 
Let $(V^{G^*},\langle - ,-\rangle,A(q))$ represent this isomorphism class in $\EuD_n$: it is isomorphic to the object from Definition \ref{defn:Dnob}, by Lemma \ref{lem:avhsDn}.

Now, up to multiplication by an overall sign, there is a unique basis $\{e_0,e_2,\ldots,e_{2n}\}$ for $V^{G^*} \cong H^{ev}(X;\C)^{G^*}$ such that
\begin{itemize}
\item $e_{i+2} = A(0) \cdot e_i$;
\item $\langle e_0,e_{2n}\rangle = (-1)^{n(n+1)/2}\iii^n\int_X \omega^n$.
\end{itemize}
This coincides with the basis $\{e,\omega,\omega^{\cup 2}, \ldots,\omega^{\cup n} \}$ for $H^{ev}(X)^{G^*}$, up to an overall sign. 
In particular, the matrix entries of $A(Q)$ with respect to this basis can be extracted from the isomorphism class in $\EuD_n$. 
They correspond to three-point, genus-zero Gromov-Witten invariants with insertions on cohomology classes $\omega,\omega^j,\omega^k$ for any $j,k$. 

\begin{rmk}
Note that these Gromov-Witten invariants are all non-negative: in particular, if there is some such Gromov-Witten invariant that does not vanish and has odd degree, we can use it to fix the sign ambiguity $Q \mapsto -Q$.  
We can do this, in particular, for the quintic $X^5$.
By comparison with classical mirror symmetry \cite[\S 6.3.3]{coxkatz}, the result in those cases is that the mirror map in Theorem \ref{thm:hmscy} is $\psi^*(q) = Q + \mathcal{O}(Q^2)$: i.e., the undetermined sign is $+1$. 
We conjecture that the sign is always $+1$.
\end{rmk}

For the quintic, we have 
\begin{eqnarray*}
A(Q) \cdot e_0 &=& e_2,\\
A(Q) \cdot e_2 &=& g(Q) \cdot e_4,\\
A(Q) \cdot e_4 &=& e_6.
\end{eqnarray*}
Thus, only a single matrix entry contains non-trivial information, namely $g(Q)$: if $[\omega] = H$ (where $H$ is Poincar\'{e} dual to the hyperplane class), we have
\begin{eqnarray*}
g(Q) \cdot \int_X \omega^3 &=& \langle [\omega],[\omega],[\omega] \rangle_{0,3}\\
\Rightarrow 5 \cdot g(Q) &=& 5 + \sum_{d =1}^\infty n_d \cdot d^3 \cdot \frac{Q^{d}}{1-Q^{d}},
\end{eqnarray*}
where $n_d$ is the virtual number of degree-$d$ rational curves on $X^5$ (see, e.g., \cite[\S  2.1]{coxkatz}).
In particular, we can compute the curve counts $n_d$ from the isomorphism class of $\EuH^A(X^5)^{G^*}$ in $\EuC_3$, hence also from the isomorphism class of $\EuH^B(\tilde{Y}^5)^G$.

In practice, this is not necessarily the most efficient way of extracting Gromov-Witten invariants: $\EuH^B(Y)$ can be efficiently computed by computing the Picard-Fuchs differential equation. 
Then the mirror map $\psi$ can be computed in terms of the first two logarithmic solutions of the Picard-Fuchs equation (which can sometimes be written in terms of hypergeometric functions), and the Yukawa coupling can also be computed by solving a certain differential equation. 
We refer the reader to \cite[Chapter 2]{coxkatz} for an explanation of these matters. 
We content ourselves with an exposition of what information about Gromov-Witten invariants can \emph{in principle} be extracted from Hodge-theoretic mirror symmetry.

\begin{rmk}
We observe that the version of Hodge-theoretic mirror symmetry in Definition \ref{defn:hodgems} is not necessarily a consequence of the version of mirror symmetry proved for Calabi-Yau complete intersections in toric varieties in \cite{Givental1996}. 
Namely, because Givental computed Gromov-Witten invariants by localization on the space of stable maps into the ambient toric variety, he computes Gromov-Witten invariants with insertions from cohomology classes restricted from the ambient variety. 
In contrast, Definition \ref{defn:hodgems} takes into account all of the cohomology of $X$, not just the ambient classes. 
However, for Calabi-Yau hypersurfaces in projective space, quantum cup product of $[\omega]$ with primitive classes is necessarily trivial, so this does not give us any non-trivial information about Gromov-Witten invariants. 
\end{rmk}

\section{The Fukaya category}
\label{sec:fuk}

Let $X$ be a connected $2n$-dimensional integral Calabi-Yau symplectic manifold, as in \S  \ref{subsec:fukint}. 
Let $\EuF(X)$ be a version of the Fukaya category of $X$. 
In this section, we give a list of properties that we need the Fukaya category $\EuF(X)$ to have in order for our results to work. 

Firstly, we need the Fukaya category to be a $\BbK_A$-linear and $\Z$-graded $\ainf$ category, where $\BbK_A := \C\laurents{Q}$.
Secondly, we need it to satisfy all of the properties enumerated in \cite[\S  2]{Perutz2015}: these will be proven for the relative Fukaya category in \cite{Perutz2015a}. 

We will not repeat all of those properties here, but recall that one of the required properties is the existence of the \emph{closed-open map}, which is a map of graded $\BbK$-algebras:
\[ \EuC\EuO: \QH^\bullet(X) \to \HH^\bullet(\EuF(X)),\]
and another is the \emph{open-closed map}, which is a map of graded $\QH^\bullet(X)$-modules:
\[ \EuO\EuC: \HH_\bullet(\EuF(X)) \to \QH^{\bullet+n}(X)\]
(here, $\HH_\bullet(\EuF(X))$ acquires a $\QH^\bullet(X)$-module structure via the closed-open map $\EuC\EuO$, and its natural $\HH^\bullet(\EuF(X))$-module structure).

Thirdly, we need the Fukaya category to satisfy some additional properties, which we list in the remainder of this section. 
These properties will be proven for the relative Fukaya category in \cite{Ganatra2015a}.

\subsection{Cyclic open-closed map}
\label{subsec:occyc}

Recall the various flavours of cyclic homology of an $\ainf$ category $\EuC$: $\HC^{+,-,\infty}_\bullet(\EuC)$ is a $W^{+,-,\infty}$-module, where $W^\infty = \BbK\laurents{u}$, $W^- = \BbK\power{u}$, $W^+ = W^\infty/W^-$ ($\HC_\bullet^\infty$ is also denoted $\HP_\bullet$, and called `periodic cyclic homology').

For the relative Fukaya category, there exist maps
\[ \widetilde{\EuO\EuC}^{-,+,\infty}: \HC_\bullet^{-,+,\infty}(\EuF(X)) \to \QH^{\bullet+n}(X)\otimes W^{+,-,\infty} \]
and these maps are compatible with the Connes periodicity exact sequences.

\begin{rmk}
In the setting of Liouville manifolds, the cyclic open-closed maps will be 
constructed (from cyclic homology of the wrapped and compact Fukaya categories to
$S^1$-equivariant symplectic cohomology and ordinary homology respectively) in \cite{Ganatra2015}. 
\end{rmk}

\subsection{Getzler-Gauss-Manin connection}
\label{subsec:GGMOC}

The negative cyclic open-closed map respects connections:
\[ \widetilde{\EuO\EuC}^{-} \circ \nabla^{GGM}_v = \nabla^{QDE}_v \circ \widetilde{\EuO\EuC}^{-},\]
where 
\[\nabla^{GGM}: T \cM_A  \otimes \HC^-_\bullet(\EuF(X))\to u^{-1} \HC^-_\bullet(\EuF(X))\]
 is the Getzler-Gauss-Manin connection (see \cite{Getzler1993}, or \cite{Sheridan2015a} for an exposition adapted to the present setting), and $\nabla^{QDE}$ is the \emph{quantum differential equation} of Definition \ref{defn:avhs}.

\subsection{Mukai pairing}
\label{subsec:OCMuk}

Because $\EuF(X)$ is proper, its Hochschild homology admits the \emph{Mukai pairing} (see \cite{Shklyarov2012} for the $\mathsf{dg}$ case, \cite{Sheridan2015a} for the $\ainf$ case):
\[ \langle-,-\rangle_{Muk}: \HH_\bullet(\EuF(X)) \otimes \HH_\bullet(\EuF(X)) \to \BbK.\]
The open-closed map intertwines the Mukai pairing with the intersection pairing on quantum cohomology:

\begin{equation}
\label{eqn:ocmuk}
\int_X \EuO \EuC(\alpha) \cup \EuO \EuC(\beta) = (-1)^{n(n+1)/2}\langle \alpha,\beta\rangle_{Muk}.
\end{equation}

\begin{example}
If $\alpha = e_{L_0}$ and $\beta = e_{L_1}$ are Chern characters of objects $L_i$, then 
\[ \langle e_{L_0},e_{L_1} \rangle_{Muk} = \chi(\Hom^\bullet(L_0,L_1))\]
(see e.g., \cite[Ex. 5.23]{Sheridan2015a}). 
If $L_i$ are objects of the Fukaya category $\EuF(X)$, they correspond to oriented Lagrangian submanifolds of $X$: and in certain situation (e.g., when $L_i$ bound no non-constant holomorphic discs) one can prove that $\EuO\EuC(e_{L_i}) = [L_i]$. 
Then \eqref{eqn:ocmuk} reduces to the well-known formula
\[ [L_0] \cdot [L_1] = (-1)^{n(n+1)/2} \chi(HF^\bullet(L_0,L_1)).\]
\end{example}

\subsection{Higher residue pairing}
\label{subsec:higherresoc}

The Mukai pairing admits a lift to negative cyclic homology, called the \emph{higher residue pairing} (see \cite{Shklyarov2013} for the $\mathsf{dg}$ case, \cite{Sheridan2015a} for the $\ainf$ case):
\[ \langle-,-\rangle_{res}: \HC^-_\bullet(\EuF(X)) \times \HC^-_\bullet(\EuF(X)) \to \BbK\power{u},\]
which is $\BbK\power{u}$-sesquilinear, and extends the Mukai pairing.
Similarly, quantum cohomology admits a sesquilinear pairing, given by the intersection pairing (see Definition \ref{defn:avhs}).

The negative cyclic open-closed map intertwines these pairings:
\[ \langle \alpha, \beta \rangle_{res} = \left\langle \widetilde{\EuO\EuC}^-(\alpha),\widetilde{\EuO\EuC}^-(\beta) \right\rangle.\]

\section{Proofs}
\label{sec:pf}

\begin{thm}[Theorem \ref{thm:cycoc}]
\label{thm:ocvhs}
The negative cyclic open-closed map
\[ \widetilde{\EuO \EuC}^-: \HC_\bullet^-(\EuF) \to \EuH^A(X)\]
is a morphism of polarized pre-\vhs.
\end{thm}
\begin{proof}
The content of \S  \ref{subsec:GGMOC} is that $\widetilde{\EuO\EuC}^-$ respects connections, hence is a morphism of unpolarized pre-\vhs; the content of \S \ref{subsec:higherresoc} is that $\widetilde{\EuO\EuC}^-$ respects polarizations.
\end{proof}

\begin{thm}[Theorem \ref{thm:ociso}]
\label{thm:ociso2}
If $\EuF(X)$ is non-degenerate and smooth, then the following maps are all isomorphisms: $\EuO\EuC$, $\EuC\EuO$, $\widetilde{\EuO\EuC}^+$, $\widetilde{\EuO\EuC}^-$, $\widetilde{\EuO\EuC}^{\infty}$.
\end{thm}
\begin{proof}
First we prove the result for $\EuO\EuC$.  
$\EuO\EuC$ contains the identity in its image by definition of non-degeneracy, and it is a map of $\QH^\bullet(X)$-modules by \cite[\S  2.4]{Perutz2015}, hence it is surjective. 
We prove that it is injective: suppose to the contrary that $\alpha \neq 0$ and $\EuO\EuC(\alpha) = 0$. 
Because $\EuF(X)$ is smooth, the Mukai pairing is non-degenerate by \cite[Theorem 1.4]{Shklyarov2012}.
Hence, there exists $\beta \in \HH_\bullet(\EuF(X))$ such that $\langle \alpha,\beta \rangle_{Muk} \neq 0$. 
By the result of \S  \ref{subsec:OCMuk}, it follows that $\langle \EuO\EuC(\alpha),\EuO\EuC(\beta) \rangle \neq 0$, hence $\EuO\EuC(\alpha) \neq 0$. 
This is a contradiction, so $\EuO\EuC$ is an isomorphism. 

It follows immediately that $\EuC\EuO$ is an isomorphism by the result of \cite[\S  2.5]{Perutz2015}, which shows that $\EuC\EuO$ is dual to $\EuO\EuC$, up to natural identifications of their domains and codomains.

It also follows immediately that $\widetilde{\EuO\EuC}^{+,-,\infty}$ are isomorphisms, by a comparison argument for the spectral sequences induced by their respective Hodge filtrations.
\end{proof}
\begin{rmk}
    The methods of \cite{Ganatra2013} (which were written for the wrapped
    Fukaya category), if developed in the setting of the relative Fukaya
    category, would give an alternate proof of this Theorem requiring only
    non-degeneracy of $\EuF(X)$. In particular, those methods show that
    smoothness of $\EuF(X)$ is a consequence of non-degeneracy, and hence a
    redundant hypothesis.  

    In the setting as above that $\EuF(X)$ is a priori {\em proper} as well
    as smooth, the existence and non-degeneracy of the Mukai pairing allows for
    the above simplified proof.  See also \cite{Abouzaid2012}.
\end{rmk}

\section{Mirror symmetry and Calabi-Yau structures} \label{sec:cystructures}

It is an idea first articulated by Kontsevich, and studied by Costello
\cite{Costello2007}, that an $\ainf$ category $\EuC$ equipped with a type of
cyclically symmetric duality called a {\em Calabi-Yau structure} should
determine a two-dimensional chain level topological field theory which attaches 
$\HH_{\bullet}(\EuC)$ to the circle, with operations controlled by chains on the (open, or
uncompactified) moduli space of punctured curves equipped with asymptotic
markers at each puncture.  Further, Calabi-Yau structures are the
first piece of input-data for a program to reconstruct the structure of an entire
cohomological field theory on $\HH_{\bullet}(\EuC)$, with operations controlled
by Deligne-Mumford compactified moduli space---see for instance
\cite{Kontsevich:Lefschetz2008} for a discussion, and \cite{Costello2009} for
related work.

In particular, suppose we have proved HMS: so we know there is a quasi-equivalence between the derived Fukaya category of $X$ and the derived category of coherent sheaves of $Y$. 
If we want to recover an isomorphism
of (closed string) cohomological field theories, we need to know which Calabi-Yau structures correspond under this quasi-equivalence. 
In this section, we explain how our Theorem \ref{thm:main} allows us to determine which Calabi-Yau structures match up under mirror symmetry; see Theorem \ref{thm:cystructures} below.

We make use below of definitions of and results about Calabi-Yau structures
developed by Konstevich-Soibelman and Konstevich-Vlassopoulous
\cite{Kontsevich2006, Kontsevich:uq}, and the categorical generalizations which
have been defined and studied in work of the first-named
author, in part joint with R. Cohen \cite{Ganatra2015, Cohen2015}.

\subsection{Smooth and proper Calabi-Yau (CY) structures}
It is now understood that
there are two types of Calabi-Yau structures, ones associated to {\it proper}
categories and ones associated to {\it smooth} categories. The chain
level 2-dimensional topological field theories which are associated to
Hochschild homology in either case are necessarily incomplete, but in different
respects: only operations with $\geq 1$ inputs or $\geq 1$ outputs respectively
are allowed \cite{Kontsevich2006, Kontsevich:uq} (these are sometimes called
`left positive' and `right positive' theories). For instance, the Hochschild
homology of a smooth, non-proper Calabi-Yau category does not admit trace maps
or pairings.

When $\EuC$ is both smooth and proper, it is a folk result that these two types
of Calabi-Yau structures are equivalent; see Proposition
\ref{smoothpropercyprop}. Moreover, in this case, Hochschild homology admits
operations as above with no restrictions on inputs or outputs. More broadly, it
is expected that a smooth and proper $\ainf$ category equipped with (either type of)
Calabi-Yau structure 
should be precisely the data required to determine an associated {\it
2-dimensional oriented extended field theory} in the sense of the Baez-Dolan
cobordism hypothesis \cite{Lurie:2009fk} (note for instance that Costello's
theorem \cite{Costello2007} also associates a partial extended, or {\it
open-closed} theory).

We use without detailed exposition the $\mathsf{dg}$ category
$[\EuC,\EuC]$ of $\ainf$ $\EuC\!-\!\EuC$ bimodules, for which there are now
many references (see e.g., \cite{Seidel2008c, Tradler2008, Ganatra2013,
Sheridan2015a}). We denote the (necessarily derived) morphism spaces in this
category by $\hom^{\bullet}_{\EuC\!-\!\EuC}$, and use the notation $ -
\otimes_{\EuC} -$ to refer to (derived) tensor product. There are
several canonical bimodules of particular interest:
\begin{itemize}
    \item The {\em diagonal bimodule} $\EuC_{\Delta}$ associates to a pair of
        objects $A,B$ the chain complex $\EuC_{\Delta}(A,B) =
        \hom_{\EuC}(A,B)$.

    \item For any pair of auxiliary objects $(K,L)$ the {\em Yoneda bimodule}
    $\EuY^l_K \otimes \EuY^r_L$ associates to a pair $(A,B)$
    the chain complex $\EuY^l_K \otimes \EuY^r_L(A,B):=
    \hom_{\EuC}(A, K) \otimes \hom_{\EuC}(L, B)$.

\item For any bimodule $\EuB$, the {\em proper (or linear) dual}
    $\EuB^{\vee}$ is, as a chain complex, the linear dual
    $\EuB^{\vee}(X,Y) := \hom_{\BbK}(\EuB(X,Y), \BbK)$ (see e.g.,
    \cite{Tradler2008} for the case of $\ainf$ algebras). If $\EuB$ is
    {\em proper}, meaning its cohomology groups $H^{\bullet}(\EuB(A,B))$ are
    finite-rank for any $A,B$, then $\EuB^{\vee}$ is proper too. We abbreviate
    $\EuC^{\vee}:= \EuC_{\Delta}^{\vee}$.

\item For any bimodule $\EuB$, the {\em smooth (or bimodule) dual}
    $\EuB^{!}$ is, as a chain complex 
     \begin{equation}
            \EuB^!(K,L) := \hom_{\EuC\!-\! \EuC}(\EuB, \EuY^l_K \otimes \EuY^r_L) \simeq \HH^{\bullet}(\EuC, \EuY^l_{K} \otimes \EuY^r_L).
        \end{equation}
    In the case of an ordinary (or $\mathsf{dg}$) bimodule $B$ over an ordinary/$\mathsf{dg}$ algebra $A$ one defines $B^! :=
    \hom_{A\!-\! A} (B, A \otimes A^{op})$ where the (derived) hom is taken using
    the outer bimodule structure on $A \otimes A^{op}$, and the bimodule structure
    on $A^!$ is induced from the inner bimodule structure on $A \otimes A^{op}$.
    For an $\ainf$ category, there is a similar explicit definition of the bimodule structure on
    $\EuB^!$, see \cite[Def. 2.41]{Ganatra2013}. 
       If $\EuB$ is perfect, meaning it is split-generated by Yoneda bimodules,
    then $\EuB^!$ is too.  
    Again we abbreviate $\EuC^!:= \EuC_{\Delta}^!$. 

\end{itemize}
\begin{rmk}
    In the literature, $\EuC^{\vee}$ and $\EuC^!$ are sometimes referred to as the {\em Serre} and {\em inverse Serre bimodules}, respectively.
\end{rmk}
Recall for what follows that positive and negative cyclic homology groups come
equipped with natural maps from and to Hochschild homology
\begin{align}
    pr: \HH_{\bullet}(\EuC) &\to \HC^+_{\bullet}(\EuC) \\
    i: \HC^-_{\bullet}(\EuC) &\to \HH_{\bullet}(\EuC).
\end{align}
These maps, which are models of the {\it projection onto (homotopy) orbits}
and {\it inclusion of (homotopy) fixed points} of an $S^1$ action, admit simple
chain-level descriptions: $pr(\alpha) = \alpha \cdot u^0$ and
$i(\sum_{j=0}^{\infty} \alpha_j u^j) = \alpha_0$. 

\begin{lem}
    If $\EuC$ is a proper $\ainf$ category over $\BbK$, there is an isomorphism
    between the linear dual of Hochschild homology and the space of bimodule
    morphisms from the diagonal bimodule $\EuC_{\Delta}$ to the {\em Serre bimodule}:
    \begin{equation}
        \label{properhhdual}
        \HH_\bullet(\EuC)^{\vee} = \hom_{\BbK}(\HH_\bullet(\EuC),\BbK) \stackrel{\sim}{\rightarrow} \hom_{\EuC\!-\!\EuC}(\EuC_{\Delta},\EuC^{\vee})
    \end{equation}
\end{lem}
\begin{proof}
    On the level of finite dimensional algebras the isomorphism is the
    canonical equivalence 
    $\hom_{A \otimes A^{op}}(A, A^{\vee}) := \HH^{\bullet}(A, A^{\vee}) \cong
    \HH_{\bullet}(A, A)^{\vee}$.  
    Similarly, there are straightforward chain-level descriptions of the isomorphism
    \eqref{properhhdual} for $\ainf$ categories $\EuC$.  
\end{proof}
\begin{defn}\label{properCY}
    If $\EuC$ is proper, an element $\phi \in \HH(\EuC)^{\vee}[-n]$ is said to
    be \emph{non-degenerate} if the corresponding morphism $\tilde{\phi} \in
    H^{0}(\hom_{\EuC\!-\!\EuC}(\EuC_{\Delta}, \EuC^{\vee}[-n]))$ is an isomorphism of bimodules.
    Equivalently, for any objects $K$ and $L$, the pairing
    \[
        \mathrm{Hom}_{\EuC}^\bullet(K,L) \otimes \mathrm{Hom}_{\EuC}^{n-\bullet}(L,K) \stackrel{\mu^2}{\to} \mathrm{Hom}^n_{\EuC}(K,K) \to \HH_{n}(\EuC) \stackrel{\phi}{\to} \BbK
    \]
    is non-degenerate.

    Let $\EuC$ be a proper $\ainf$ category. A \emph{weak proper Calabi-Yau
    (CY) structure} of dimension $n$ is a non-degenerate morphism
    $\phi: \HH(\EuC) \rightarrow \BbK$ of degree $-n$, or equivalently (the
    cohomology class of) a bimodule quasi-isomorphism $\phi: \EuC_{\Delta} \to
    \EuC^{\vee}[-n]$.\\

    A \emph{(strong) proper Calabi-Yau} structure is a morphism
    $[\tilde{\phi}]: \HC^+(\EuC) \rightarrow \BbK$ such that the composition
    $[\phi] = [\tilde{\phi}] \circ pr: \HH(\EuC) \rightarrow \BbK$ is a
    weak proper Calabi-Yau structure.
\end{defn}
A \emph{proper Calabi-Yau category} is a proper $\ainf$ category equipped with a
(strong) proper Calabi-Yau structure.

\begin{rmk}
Sometimes the word `compact' is used instead of `proper'.
\end{rmk}

\begin{rmk} 
    A closely related notion which appears in Costello's work
    \cite{Costello2007, Costello2009} is that of a {\it cyclic $\ainf$
    structure}; this is an $\ainf$ category equipped with a non-degenerate
    pairing on morphism spaces such that the induced correlation functions
    $\langle \mu^k(-, \ldots, -), - \rangle$ are strictly symmetric.
    Kontsevich-Soibelman \cite[Thm. 10.2.2]{Kontsevich2006} proved that in characteristic 0,
    any proper Calabi-Yau category is quasi-isomorphic to a (unique isomorphism
    class of) cyclic $\ainf$ category; in this sense Definition \ref{properCY}
    is a homotopical relaxment of the strict cyclicity condition. 
    
    Away from characteristic 0, cyclic $\ainf$ structures and Definition
    \ref{properCY} are very different, and it seems that the latter notion,
    involving cyclic homology is the correct notion (for instance, when the
    Fukaya category is defined over a non-characteristic zero field, it carries
    a strong proper Calabi-Yau structure \cite{Ganatra2015}).
\end{rmk}

There is an alternate notion of Calabi-Yau structure for a smooth, but not
necessarily proper category $\EuC$, due to Kontsevich and Vlassopolous
\cite{Kontsevich:uq} (see also
\cite{Kontsevich2006, Ginzburg:2007fk} for a weak version
of this structure, without homotopy-cyclic invariance).  
\begin{lem} \emph{(\cite[Remark 8.2.4]{Kontsevich2006}, \cite{Cohen2015})}
If $\EuC$ is a smooth $\ainf$ category, there is an isomorphism between
Hochschild homology and the space of (derived) bimodule morphisms from the
{\em inverse Serre bimodule} to the diagonal bimodule $\EuC_{\Delta}$:
\begin{equation}\label{smoothhh}
    \HH_{\bullet}(\EuC) \cong \hom^{\bullet}_{\EuC\!-\!\EuC}(\EuC^!, \EuC_{\Delta}) = \hom^0(\EuC^!, \EuC_{\Delta}[\bullet]).
\end{equation}
\end{lem}
\begin{proof}
    For $\mathsf{dg}$ algebras, this follows from two natural maps which are
    both equivalences for any perfect bimodules $B$ and $P$ (in particular for
    $B = P = A_{\Delta}$ if $A$ is smooth): $B^!  \otimes_{A\!-\!A} P
    \stackrel{\sim}{\rightarrow} \hom_{A\!-\!A}(B, P)$, and $B
    \stackrel{\sim}{\rightarrow} (B^!)^!$.
\end{proof}
\begin{defn}\label{smoothcy}
    If $\EuC$ is smooth, an element $\sigma \in \HH_{-n}(\EuC)$ is said to be \emph{non-degenerate} if $\sigma$ corresponds under \eqref{smoothhh} to a bimodule quasi-isomorphism. Equivalently for any pair of objects $K, L$, capping with $\sigma$ should induce an isomorphism
    \[
    \cap \sigma: \EuC^!(K,L) = \HH^{\bullet}(\EuC, \EuY^l_{K} \otimes \EuY^r_L) \stackrel{\sim}{\to} \HH_{\bullet -n} (\EuC, \EuY^l_{K} \otimes \EuY^r_L) \cong \hom_{\EuC}(K,L).\]
    A \emph{weak smooth $n$-dimensional Calabi-Yau structure} is a non-degenerate element $[\Omega] \in \HH_{-n}(\EuC)$.
    A \emph{(strong) smooth $n$-dimensional Calabi-Yau structure} is an element
    $[\tilde{\Omega}] \in \HC^-_{-n}(\EuC)$ (or equivalently a morphism
    $[\tilde{\Omega}]: \BbK \rightarrow \HC^-(\EuC)$), such that the induced
    element of Hochschild homology $[\Omega] := i [\tilde{\Omega}]$ is a weak
    smooth $n$-dimensional Calabi-Yau structure.  
\end{defn}

When $\EuC$ is simultaneously smooth and proper, note that 
\begin{lem}\label{inversefunctors}
    The bimodules $\EuC^{\vee}$ and $\EuC^{!}$ induce mutually inverse
    endofunctors of $perf(\EuC)$.
\end{lem}
\begin{proof}
    See e.g., \cite[Prop. 20.5.5]{Ginzburg:2005aa} for a proof in the
case of $\mathsf{dg}$ algebras $A$, which essentially use both finiteness conditions on
$A$ and adjunctions to write a chain of quasi-isomorphisms of the form:
\[
    \hom_{A\!-\! A}(A, A \otimes_{\BbK} A) \otimes_A A^{\vee} \simeq \hom_{A\!-\!A}(A, A^{\vee} \otimes_{\BbK} A) \simeq \hom_{A}( (A^{\vee})^{\vee} \otimes_A A, A) \simeq A
\]
(where recall, as in the case of $\ainf$ categories and bimodules,
$\hom_{A\!-\!A}$ refers to derived bimodule morphisms, and $\otimes_A$ refers to the derived
tensor product).  The general case, for $\ainf$ categories $\EuC$, is
identical.
\end{proof}

Next, we recall the well-known fact, established by Hood and Jones, that
negative cylic homology and the $\BbK$ dual of positive cyclic homology are
$\BbK \power{u} $ dual: 
\begin{lem}[\cite{Hood:1987aa}]\label{cyclichoodjones}
There is a canonical isomorphism 
\begin{equation}
    \hom_{\BbK\power{u} } (\HC_{\bullet}^-(\EuC), \BbK \power{u} ) \cong \hom_{\BbK}(\HC^+_{\bullet}(\EuC), \BbK).
\end{equation}
\end{lem}
\begin{proof}
    Use the fact that  as $\BbK \power{u} $-modules, $\BbK \power{u} \cong \hom_{\BbK } ( \BbK ( ( u ) )  / u \BbK\power{u}, \BbK)$, coupled with the adjunction $\hom_{\BbK\power{u} } (\HC_{\bullet}^-(\EuC), \hom_{\BbK  } ( \BbK ( ( u ) )  / u \BbK\power{u}, \BbK)) \cong \hom_{ \BbK}(\HC_{\bullet}^-(\EuC) \otimes_{\BbK \power{u} }  \BbK ( ( u ) )  / u \BbK\power{u}, \BbK)$.
\end{proof}

Using this fact, one can compare (strong) smooth and proper Calabi-Yau structures.
\begin{prop} \emph{(compare \cite{Kontsevich:uq}; in the weak case, compare \cite[Prop.  3.2.4]{Ginzburg:2007fk}, \cite[Conj. 10.2.8]{Kontsevich2006})}
    \label{smoothpropercyprop}
    On a smooth and proper category $\EuC$, proper and smooth $n$-dimensional (strong)
    Calabi-Yau structures are in bijection.
\end{prop}
\begin{proof}
    The higher residue pairing induces a map 
    \begin{eqnarray*}
	\HC^-(\EuC) &\rightarrow & \hom_{\BbK \power{u} } (\HC^-(\EuC), \BbK \power{u} ) \cong \hom_{\BbK} (\HC^+(\EuC), \BbK) \\
	\alpha & \mapsto & \langle -,\alpha \rangle_{res},
\end{eqnarray*}
    where the
    isomorphism between the $\BbK\power{u} $ dual of negative cyclic homology
    and the $\BbK$ dual of cyclic homology is by
    Lemma \ref{cyclichoodjones}.
    This map is an isomorphism whenever the $u=0$ reduced map $\HH(\EuC)
    \rightarrow \HH(\EuC)^{\vee}$ is, for instance when $\EuC$ is smooth and
    proper (\cite[Theorem 1.4]{Shklyarov2012}, or see \cite{Sheridan2015a} for the
    $\ainf$ case).  So an element of $\HC^-(\EuC)$ induces a unique morphism
    $\HC^+(\EuC) \rightarrow \BbK$ and vice versa.

    It remains to compare the non-degeneracy conditions for smooth and proper
    Calabi-Yau structures. This is an immediate consequence of the following
    fact: under the identifications $\hom_{\EuC\!-\!\EuC}(\EuC^!,
    \EuC_{\Delta}) \cong \HH(\EuC)$, $\hom_{\EuC\!-\!\EuC}(\EuC_{\Delta},
    \EuC^{\vee}) \cong \HH(\EuC)^{\vee}$, the partial adjoint of the Mukai
    pairing, mapping $\HH(\EuC) \rightarrow \HH(\EuC)^{\vee}$ is cohomologically
    equivalent to the map on morphisms induced by convolving with
    $\EuC^{\vee}$:
    \begin{equation}
        \hom_{\EuC\!-\!\EuC}(\EuC^!, \EuC_{\Delta}) \stackrel{ \cdot \otimes_{\EuC} \EuC^{\vee}}{\rightarrow} \hom_{\EuC\!-\!\EuC}(\EuC^! \otimes_{\EuC} \EuC^{\vee}, \EuC_{\Delta} \otimes_{\EuC} \EuC^{\vee}) \simeq \hom_{\EuC\!-\!\EuC}(\EuC_{\Delta},\EuC^{\vee}).
    \end{equation}
    By Lemma \ref{inversefunctors} $\EuC^{\vee}$ is invertible, and in
    particular it sends bimodule quasi-isomorphisms to bimodule
    quasi-isomorphisms (which are the relevant notions of non-degenerate in
    each case).
\end{proof}

Given a smooth category $\EuC$, the inclusion $\EuC \hookrightarrow \twsplit
\EuC$ is a Morita equivalence, and hence induces an isomorphism on
$\HC_{\bullet}^{+/-/\infty}$.  This isomorphism unsurprisingly can be shown to
preserve non-degenerate elements, so one can unambiguously talk about
Calabi-Yau structures on $\EuC$ or $\twsplit \EuC$:
\begin{prop}[`Morita invariance of Calabi-Yau structures', see \cite{Cohen2015}]
If $\EuC$ is smooth (resp. proper), the inclusion $\EuC \hookrightarrow
\twsplit \EuC$ induces an isomorphism of spaces of smooth (resp. proper)
Calabi-Yau structures.  
\end{prop}
The above Proposition serves as a motivation for the following definition: 
\begin{defn}\label{smoothcyequivalence}
Let $\EuC$ and $\EuD$ be smooth and proper $\ainf$ categories equipped with
smooth Calabi-Yau structures $[\tilde{\Omega}_{\EuC}]$,
$[\tilde{\Omega}_{\EuD}]$. We say that $\EuC$ and $\EuD$ are {\it Calabi-Yau (Morita)
equivalent} if there is an $\ainf$ quasi-equivalence $\mathcal{F}: \twsplit\EuC
\rightarrow \twsplit\EuD$ sending $[\tilde{\Omega}_{\EuC}]$ to $[\tilde{\Omega}_{\EuD}]$. 
\end{defn}
\begin{rmk}\label{rmk:equivalenceproperCY}
    There is an obvious notion of Calabi-Yau equivalence for proper Calabi-Yau
    structures, namely that the induced map $\mathcal{F}^*$ from
    $\hom_{\BbK}(\HC^+(\EuD),\BbK)$ to $\hom_{\BbK}(\HC^+(\EuC), \BbK)$ should
    preserve proper Calabi-Yau elements. The methods of Proposition
    \ref{smoothpropercyprop} imply that under the correspondence between proper
    and smooth Calabi-Yau structures, this notion is equivalent to that of
    Definition \ref{smoothcyequivalence}.
\end{rmk}

Going forward, we may sometimes refer to the equivalent data of a smooth or
proper CY structure on a smooth and proper category $\EuC$ as simply a
\emph{Calabi-Yau structure}.

\subsection{Geometric Calabi-Yau structures}

Suppose now that the hypotheses of Theorem \ref{thm:main} hold for a pair $(X,
Y)$.

We observe, following \cite{Ganatra2015}, that the Fukaya category $\EuF$ of $X$
comes equipped with a canonical geometric $n$-dimensional (strong) proper
Calabi-Yau structure, which can be characterized using the cyclic open-closed
map.
Specifically, one defines
\begin{equation}
    \widetilde{\phi}_{\EuF}: \HC^+(\EuF) \to \BbK[-n]
\end{equation}
to correspond to 
\begin{equation}\label{geometricproperCYstructure}
    \left( \widetilde{\EuO \EuC}^-( - ) ,  \tilde{e}\right):
    \HC^-(\EuF) \to \BbK\power{u} [ -n] 
\end{equation}
via Lemma \ref{cyclichoodjones}, where $\tilde{e} = e \cdot u^0 \in H^\bullet(X;
\C) \otimes_{\C} \BbK_A \power{u} [n]$, and $( \cdot , \cdot )$ is the pairing
described in Definition \ref{defn:avhs}.  Concretely, if $\widetilde{\EuO\EuC} = \sum_{i \ge 0} u^i\cdot \widetilde{\EuO\EuC}_i$, one extends by linearity the following map on basic generators:
\begin{equation}
    \widetilde{\phi}_\EuF\left(u^{-i}\cdot \alpha \right) := (-1)^{n(n+1)/2} \int_X \widetilde{\EuO \EuC}_i(\alpha).
\end{equation}
So this is picking out the portion of $\widetilde{\EuO \EuC}_i
(\alpha)$ that is hitting the top class in $u^0 \cdot H^{2n}(X)$.  The fact
that the induced map $\phi = pr \circ \tilde{\phi}: \HH_{\bullet}(\EuC)
\rightarrow \BbK$ is
non-degenerate is a well-known consequence of Poincar\'{e}-Floer
duality for Floer cohomology of compact Lagrangians: the induced pairing
\[ HF^\bullet(K,L) \otimes HF^{n-\bullet}(L,K) \to \BbK\] is equal to the usual
non-degenerate pairing on Floer cohomology (see, e.g., \cite[Lemma 2.4]{Sheridan2013};
different geometric technical hypotheses are assumed, but the proof carries
over verbatim). 

By Proposition \ref{smoothpropercyprop}, since $\EuF$ is smooth and
proper, there is a unique smooth CY structure corresponding to
$\tilde{\phi}_\EuF$. It is convenient as in \cite{Ganatra2015} to characterize
this smooth CY structure via the negative cyclic open-closed
map as follows: since the negative cyclic open-closed map $\widetilde{\EuO
\EuC}^-$ is an isomorphism by Theorem \ref{thm:ociso}, there is a unique
element $\tilde{\Omega} \in \HC^-_{-n}(\EuF)$ with $\widetilde{\EuO
\EuC}^-(\tilde{\Omega}) = \tilde{e}$. Since $\widetilde{\EuO \EuC}^-$ intertwines
polarizations by Theorem \ref{thm:cycoc}, we have
\begin{equation}
    \langle  -,\tilde{\Omega}\rangle_{res} = \left(\widetilde{\EuO \EuC}^-( - ), \tilde{e} \right) = \widetilde{\phi}_\EuF.
\end{equation}   
Therefore, $\langle
-,\tilde{\Omega}\rangle_{res}$ corresponds to $\tilde{\phi}_{\EuF}$ via
the isomorphism of Lemma
\ref{cyclichoodjones}. By Proposition \ref{smoothpropercyprop}, it follows that
$\tilde{\Omega}$ gives the desired smooth CY structure on $\EuF$.
\begin{rmk}
    The proof that an element $\tilde{\Omega}$ satisfying $\widetilde{\EuO
    \EuC}^-(\tilde{\Omega}) = \tilde{e}$ determines a geometric smooth CY structure for
    $\EuF$ was first given in \cite{Ganatra2013, Ganatra2015} 
    in the setting of the wrapped Fukaya category (which is not proper, hence
    does not admit a Mukai or residue pairing). While the methods implemented
    there could also be implemented in the setting of a compact Calabi-Yau
    target, the presence of the Mukai and residue pairings in this smooth and
    proper situation allows for a simplified proof (using Proposition
    \ref{smoothpropercyprop}).
\end{rmk}

By the discussion in \S \ref{subsec:QHnorm}, the element $\tilde{e}$ is
Hodge-theoretically {\it normalized}, in the sense of Definition
\ref{defn:normOm}. Hence, as $\widetilde{\EuO \EuC}^-$ is an isomorphism of \vhs{} 
by Theorems \ref{thm:cycoc} and \ref{thm:ociso}, and isomorphisms of \vhs{} preserve the complex vector
space of normalized volume forms, we see that
\begin{cor}\label{fukayacynormalized}
The canonical smooth Calabi-Yau structure on the Fukaya category is
Hodge-theoretically normalized.  
\end{cor}

Now we turn to the mirror family $Y \to \cM_B$. 
There is a unique Calabi-Yau structure on $\dbdg{Y}$ which corresponds to the canonical Calabi-Yau structure on $\EuF(X)$ under HMS. 
We have the isomorphism
\[ \tilde{I}_K: \HC_{-n}^-(\dbdg{Y}) \to H^0(\Omega^n_{Y/\cM_B}),\]
so we see that $n$-dimensional smooth Calabi-Yau structures on $\dbdg{Y}$ are in one-to-one correspondence with non-vanishing classes $\Omega_{Y}$ in the one-dimensional $\BbK_B$-vector space $H^0(\Omega^n_{Y/\cM_B})$. 
In particular, an $n$-dimensional Calabi-Yau structure corresponds to a choice of relative volume form for the family $Y\to \cM_B$. 
We recall the discussion in \S \ref{subsec:app}: $\Omega_Y$ determines a trivialization of the canonical bundle, and hence a pairing 
\[
\ext^\bullet(\mathcal{E},\mathcal{F}) \otimes
\ext^{n-\bullet}(\mathcal{F},\mathcal{E}) \to \BbK_B\] by Serre duality, which
corresponds to the Poincar\'{e} duality pairing on Lagrangian Floer cohomology,
under mirror symmetry.

We have arrived at the central question of this section: which relative volume form $\Omega_Y$ corresponds to the Calabi-Yau structure on the Fukaya category $\EuF(X)$?
Combining Corollary \ref{fukayacynormalized} and Theorem \ref{thm:main}, we see immediately that $\Omega_Y$ must be Hodge-theoretically normalized: this reduces the $\BbK_B^*$-ambiguity in $\Omega_Y$ to a $\C^*$-ambiguity. 
We can reduce this ambiguity even further: as a consequence of Hodge-theoretic mirror symmetry, the leading-order terms of the Yukawa couplings coincide. This means
\begin{eqnarray*}
 (-1)^{n(n+1)} \int_X \omega^n &=& \left. \int_Y \Omega_Y \wedge KS(q \partial_q)^n \cdot \Omega_Y \right|_{q=0}.
\end{eqnarray*}
This allows us to fix the $\C^*$ ambiguity in $\Omega_Y$ up to a sign. We have established:

\begin{thm}\label{thm:cystructures}
    Suppose that $X$ and $Y$ satisfy the hypotheses of Theorem \ref{thm:main}.
    Then the equivalence $\twsplit \EuF(X) \cong \psi^* \dbdg{Y}$ is in fact an
    equivalence of (smooth and proper) Calabi-Yau categories, where
    \begin{itemize}
        \item $\EuF(X)$ (and hence $\twsplit \EuF(X)$) is equipped with its canonical CY structure.

        \item $\dbdg{Y}$ is equipped with the unique (up to sign) CY structure corresponding
            to the Hodge-theoretically normalized relative volume form
            $\Omega_Y$ on $Y$ whose Yukawa coupling's leading term  
            is $(-1)^{n(n+1)/2} \int_X \omega^n$.
    \end{itemize}
\end{thm}

It follows from \cite{Kontsevich2006, Kontsevich:uq} that the Hochschild
homologies of $\twsplit \EuF(X)$ and $\psi^* \dbdg{Y}$, which are isomorphic to
$\QH^\bullet(X)[n]$ and $H^\bullet(\Omega^{-\bullet}Y)$ respectively, carry induced TFT
operations parametrized by chains on the (open) moduli space of curves with
marked points equipped with asymptotic markers, and moreover that these TFT
operations are equivalent. Although these operations are not quite parametrized
by Deligne-Mumford compactified moduli spaces, some of the operations will
coincide with the compactified operations (for instance the Yukawa couplings).
One thus expects that the open-closed map intertwines these TFT operations on
$\HH_\bullet(\EuF(X))$ with operations defined in terms of the closed
Gromov-Witten invariants of $X$ -- however we have not proved this.

If that were the case, one could hope to compute some higher-genus Gromov-Witten invariants (i.e., numbers) in this way, by writing the matrix coefficients with respect to a basis. 
The natural bases, with respect to which the matrix coefficients of the TFT operations are actual Gromov-Witten invariants, are bases for $H^\bullet(X;\C) \subset H^\bullet(X;\BbK_A)$. 
We would like to know what such bases correspond to on the other side of mirror symmetry. 
This can be achieved by giving an intrinsic characterization of $H^\bullet(X;\C) \subset H^\bullet(X;\BbK_A)$ in terms of the $A$-model \vhs{} structure: indeed, $H^\bullet(X;\BbK_A)$ is the associated graded of the Hodge filtration, which is isomorphic to the associated graded of the weight filtration in our setting. 
The induced connection on the associated graded of the weight filtration is trivial; and the flat sections of this connection correspond precisely to $H^\bullet(X;\C) \subset H^\bullet(X;\BbK_A)$, as follows immediately from the formula for the quantum connection.

\bibliographystyle{amsalpha}
\bibliography{library}

\end{document}